%% file: Draft_09_28_19.tex
\documentclass[11pt]{amsart}

\usepackage{fullpage}
\usepackage{colonequals}
\usepackage[usenames]{color}
\usepackage{float}
\usepackage{enumerate}
\usepackage{tikz}
\usepackage{tikz-cd}
\usepackage{graphicx}
\usepackage{transparent}
\usepackage{amsmath, amssymb}
\usepackage{amsthm}
\usepackage{amsmath, amssymb}

\usepackage{epstopdf}

\newcommand{\intr}{\operatorname{int}}

\newcommand{\Z}{\mathbb{Z}}

\newcommand{\Mod}{\operatorname{Mod}}
\newcommand{\NS}{\mathcal{NS}}
\newcommand{\C}{\mathcal{C}}

\theoremstyle{plain}
\newtheorem{thm}{Theorem}[section]
\newtheorem{lem}[thm]{Lemma}
\newtheorem{cor}[thm]{Corollary}

\newtheorem{prop}[thm]{Proposition}

\newtheorem{claim}[thm]{Claim}

\theoremstyle{definition}
\newtheorem*{defn}{Definition}

\begin{document}

\title{Uniform hyperbolicity of the graphs of nonseparating curves via bicorn curves}
\author{Alexander J. Rasmussen}
\date{\today}

\begin{abstract}
We show that the graphs of nonseparating curves for oriented finite type surfaces are uniformly hyperbolic. Our proof follows the proof of uniform hyperbolicity of the graphs of curves for closed surfaces due to Przytycki-Sisto, while introducing new arguments using homology to certify that certain curves are nonseparating. As demonstrated by Aramayona-Valdez, this proves also that the graph of nonseparating curves for any oriented infinite type surface with finite positive genus is hyperbolic.
\end{abstract}

\maketitle

\section{Introduction}
In this paper we investigate the geometry of the graphs of nonseparating curves for oriented surfaces of finite type (i.e. with finitely generated fundamental group) and infinite type. The graph of nonseparating curves for an oriented finite type surface has been shown to be hyperbolic by Hamenst\"{a}dt \cite{ham} and a proof can also be deduced from the work of Masur-Schleimer \cite{ms}. However, in both of these proofs, it is not clear whether the constant of hyperbolicity obtained can be made independent of the topological type of the surface. Aramayona-Valdez ask in \cite{av} whether the graphs of nonseparating curves are uniformly hyperbolic (that is, whether the constant of hyperbolicity can be taken to be independent of the topological type of the surface). We prove in this paper that this is indeed the case.

\begin{thm}
\label{main}
There exists $\delta>0$ with the following property. Let $S$ be any finite-type oriented surface of positive genus. Then the graph of nonseparating curves $\NS(S)$ is $\delta$-hyperbolic and infinite diameter.
\end{thm}

Part of the motivation for studying this problem comes from the fact that the graphs of curves are uniformly hyperbolic. The graph of curves $\C(S)$ for an oriented finite type surface $S$ is a graph acted on by the mapping class group $\Mod(S)$. The vertices of $\C(S)$ are the free homotopy classes of simple closed curves on $S$ which do not bound a disk or a once-punctured disk and the edges join pairs of homotopy classes that have disjoint representatives. Here and elsewhere in the paper, graphs are considered as geodesic metric spaces after identifying their edges with the unit interval. Studying the action of $\Mod(S)$ on $\C(S)$ has proved useful for proving results on the algebra and geometry of $\Mod(S)$. Many such results rely on the fact that $\C(S)$ is \textit{Gromov hyperbolic}. Important results proven using hyperbolicity of $\C(S)$ include \textit{quasi-isometric rigidity} of $\Mod(S)$ (\cite{rigidity}), computation of the \textit{geometric rank} of $\Mod(S)$ (\cite{rank}), and \textit{SQ-universality} of $\Mod(S)$ (\cite{dgo}). Hyperbolicity of $\C(S)$ is also a crucial ingredient in the proof of the Ending Lamination Theorem for infinite volume hyperbolic 3-manifolds, due to Minsky and Brock-Canary-Minsky (\cite{models}, \cite{elc}).

Hyperbolicity of the graphs of curves was first proved by Masur-Minsky in \cite{mm}. Later, proofs of uniform hyperbolicity were given independently by Aougab \cite{aou}, Bowditch \cite{bow}, Clay-Rafi-Schleimer \cite{crs}, and Hensel-Pzytycki-Webb \cite{hpw}. The proof from \cite{hpw} is particularly short. An even shorter proof of uniform hyperbolicity was found by Przytycki-Sisto \cite{ps} for the special case of curve graphs of \textit{closed} oriented surfaces. However, this proof does not seem to extend immediately to the case of oriented surfaces with finitely many punctures. Despite this difficulty, in this paper we show that the proof of Przytycki-Sisto \textit{can} be extended to give uniform hyperbolicity of the graphs of nonseparating curves.

Another source of motivation for proving that the graphs of nonseparating curves are uniformly hyperbolic comes from the theory of \textit{big mapping class groups}, or in other words, mapping class groups of infinite type surfaces. Big mapping class groups arise naturally in many contexts and particularly in the context of dynamics as described by Calegari on his blog \cite{blog}. Research on big mapping class groups has accelerated recently, with many researchers investigating analogies with mapping class groups of finite type surfaces. Such analogies fail in many cases: big mapping class groups are uncountable, not residually finite \cite{pv}, and not acylindrically hyperbolic \cite{acyl}.

Nonetheless, many authors have recently investigated various graphs acted on by big mapping class groups, in analogy with the curve graphs of finite type surfaces: \cite{afp}, \cite{av}, \cite{ray}, \cite{dfv}, \cite{fp}. Note that the curve graph itself for an infinite type surface always has diameter two and therefore is trivial up to quasi-isometry. In his blog post \cite{blog}, Calegari examined the mapping class group of the plane minus a Cantor set. Hoping to find a graph better suited to studying this mapping class group, he defined the ray graph of the plane minus a Cantor set. In \cite{ray}, Bavard showed that this graph is hyperbolic and infinite diameter. In \cite{afp}, Aramayona-Fossas-Parlier generalized this construction for any infinite-type oriented surface with at least one isolated puncture, showing that the resulting graph is hyperbolic and infinite diameter. In \cite{dfv}, Durham-Fanoni-Vlamis constructed certain infinite diameter hyperbolic graphs for a general family of surfaces all having infinite genus or at least one isolated puncture.

The above constructions give infinite diameter hyperbolic graphs acted on by $\Mod(S)$ for certain surfaces $S$ with infinite genus or at least one isolated puncture. Desiring to circumvent this restriction, in \cite{av} Aramayona-Valdez studied the graph of nonseparating curves for a finite genus oriented surface, establishing that this graph is infinite diameter and that it is hyperbolic if and only if the graphs of nonseparating curves for oriented finite type surfaces are uniformly hyperbolic. Therefore we have the following: 

\begin{cor}
Let $S$ be an oriented infinite type surface with finite positive genus. Then $\NS(S)$ is $\delta$-hyperbolic and infinite diameter with $\delta$ as in Theorem \ref{main}.
\end{cor}

\noindent Notice that the above Corollary applies even when $S$ has isolated punctures. The importance however is that if $S$ has no isolated punctures, then $\NS(S)$ is the first naturally defined infinite diameter hyperbolic graph known to be acted on by $\Mod(S)$.

Further work needs to be done to determine the applications of the action on $\NS(S)$ to the study of $\Mod(S)$. This action is not acylindrical since infinitely many mapping classes stabilize any pair of nonseparating curves. One natural application of the action would be to investigate the existence of nontrivial quasimorphisms on $\Mod(S)$ (that is, quasimorphisms which are not at finite distance from any homomorphism). Note that $\Mod(S)$ admits certain nontrivial quasimorphisms simply because there is a quotient homomorphism $\Mod(S)\to \Mod(\overline{S})$, where $\overline{S}$ is $S$ with the punctures filled in, and $\Mod(\overline{S})$ admits many nontrivial quasimorphisms. Studying the action of $\Mod(S)$ on $\NS(S)$ may provide examples of nontrivial quasimorphisms not arising in this way.

\subsection{About the proof}
As in the proof of uniform hyperbolicity of the graphs of curves from \cite{ps} we use a ``guessing geodesics'' Lemma of Masur-Schleimer and Bowditch. This leads us to consider \textit{bicorns} between a pair of nonseparating curves $a$ and $b$. Because we are considering the graph $\NS(S)$, we must restrict to considering only nonseparating bicorns. Surprisingly, while this restriction throws away quite a few of the bicorns between $a$ and $b$, we show that the subgraph of $\NS(S)$ spanned by the nonseparating bicorns is still connected. To do this, we follow the idea of Przytycki-Sisto of extending an arc of a bicorn to produce a new bicorn. The resulting bicorn is not generally nonseparating. However, there is some flexibility in the procedure and we use tricks involving homology to show that we can produce some nonseparating curve this way. Similarly in the proof of thinness of ``triangles'' made of bicorn curves, given a nonseparating curve on one side of a triangle we use homology to find a nonseparating bicorn in the union of the other two sides. Unlike in the Przytycki-Sisto proof, this new bicorn may intersect the original bicorn many times, and we must generally do more work to explicitly produce a path of bounded length from the original bicorn to the new one. This is done by analyzing in detail the possible patterns of intersections between the two curves and again taking advantage of homology.

We focus on the case that $S$ has genus at least two in Section \ref{2plus} and then remark on the minor modifications when $S$ has genus one in Section \ref{genus1}.

\bigskip

\noindent{\bf Acknowledgements.} I would like to thank Javier Aramayona for suggesting this problem and for discussing potential applications and further questions with me. I would like to thank my advisor, Yair Minsky, for many helpful conversations. Finally I would like to thank Yair Minsky and Katie Vokes for reading drafts of this paper. I was partially supported by the NSF grant DMS-1610827.

\section{Preliminaries}

Here we recall some basic definitions from metric geometry. Let $(X,d)$ be a geodesic metric space. For $x,y\in X$ we denote by $[x,y]$ a geodesic in $X$ from $x$ to $y$, even though such a geodesic may not be uniquely determined by its endpoints. Denote by $B_r(A)$ the closed $r$-neighborhood of a subset $A\subset X$.

\begin{defn}
Let $\delta>0$. The geodesic metric space $(X,d)$ is \emph{$\delta$-hyperbolic} if for any $x,y,z\in X$ we have \[[x,z]\subset B_\delta([x,y]\cup [y,z]).\] We say that $(X,d)$ is \emph{hyperbolic} if it is $\delta$-hyperbolic for some $\delta>0$.
\end{defn}

\begin{defn}
Let $(X,d_X)$ and $(Y,d_Y)$ be geodesic metric spaces and $C,K>0$. A map $f:X\to Y$ is a \emph{$(K,C)$ quasi-isometric embedding} if for all $x,y \in X$ \[\frac{1}{K} d_X(x,y)-C \leq d_X(f(x),f(y))\leq Kd_X(x,y)+C.\] We say that $f$ is a \emph{quasi-isometry} if in addition there exists $r>0$ such that \[Y=B_r(f(X)).\]
\end{defn}

The following result is standard:

\begin{prop} \label{qiinvariance}
Let $(X,d_X)$ and $(Y,d_Y)$ be geodesic metric spaces and $f\colon X\to Y$ a $(K,C)$-quasi-isometry. Suppose that $(X,d_X)$ is $\delta$-hyperbolic. Then $(Y,d_Y)$ is $\delta'$-hyperbolic where $\delta'$ depends only on $\delta$, the quasi-isometry constants $K,C$, and the number $r$ such that $Y=B_r(f(X))$.
\end{prop}

\section{The case of genus at least two}
\label{2plus}

\subsection{Set up}
\label{setup}

Let $S$ be an oriented finite type surface of genus at least two. A priori, $S$ may have boundary components and punctures. However, since the difference between boundary components and punctures will not be relevant for us, we will assume that $S$ has no punctures. In other words we assume that $S$ is compact. Denote by $\NS(S)$ the graph of nonseparating curves defined as follows. The vertices of $\NS(S)$ are free homotopy classes of nonseparating simple closed curves in $S$. Two vertices are joined by an edge if the corresponding homotopy classes have representatives intersecting \textit{at most twice}. We will show that $\NS(S)$ is hyperbolic with hyperbolicity constant independent of $S$.

Note that the graph of nonseparating curves is usually defined by joining vertices with an edge if they have disjoint representatives. We denote this graph by $\NS'(S)$. However, Lemma \ref{surg} below easily implies that $\NS(S)$ is quasi-isometric to $\NS'(S)$ with quasi-isometry constants independent of $S$. Hence our proof will also give uniform hyperbolicity of $\NS'(S)$ by Proposition \ref{qiinvariance}.

Let $d(\cdot,\cdot)$ denote the distance in $\NS(S)$ or in $\NS'(S)$. Let $i(\cdot,\cdot)$ denote geometric intersection number. In other words, if $a$ and $b$ are free homotopy classes of simple closed curves then $i(a,b)$ is the minimal number of points of intersection between a simple closed curve representing $a$ and a simple closed curve representing $b$. For $a$ an oriented closed curve on $S$ we denote by $[a]$ the class represented by $a$ in $H_1(S,\partial S;\Z)$.

Throughout the proofs that follow we will use the following facts:
\begin{enumerate}[(i)]
\item A simple closed curve $a$ is separating if and only if $[a]=0$ (when we give $a$ either orientation).
\item A simple closed curve $a$ is nonseparating if and only if there exists a curve $b$ such that $i(a,b)=1$.
\end{enumerate}

\begin{lem}
For $a,b\in\NS'(S)^0$, we have $d_{\NS'(S)}(a,b)\leq 2i(a,b)+1$.
\label{surg}
\end{lem}

\begin{proof}
The proof is a slight variation of a proof due to Masur-Minsky \cite{mm}.

We prove the statement by induction on $i(a,b)$. If $i(a,b)=1$ then a regular neighborhood of $a\cup b$ is a once-punctured torus. Hence the complement of this neighborhood has positive genus and therefore there is some non-separating curve $c$ of $S$ contained in the complement of $a\cup b$. Therefore actually $d(a,b)=2$ in this case.

If $i(a,b)\geq 2$ we may take $a$ and $b$ to be in minimal position. Orient $a$ and consider two points of $a\cap b$, consecutive along $b$.

\begin{figure}[h]
\begin{tabular}{c c}
\def\svgwidth{100pt}
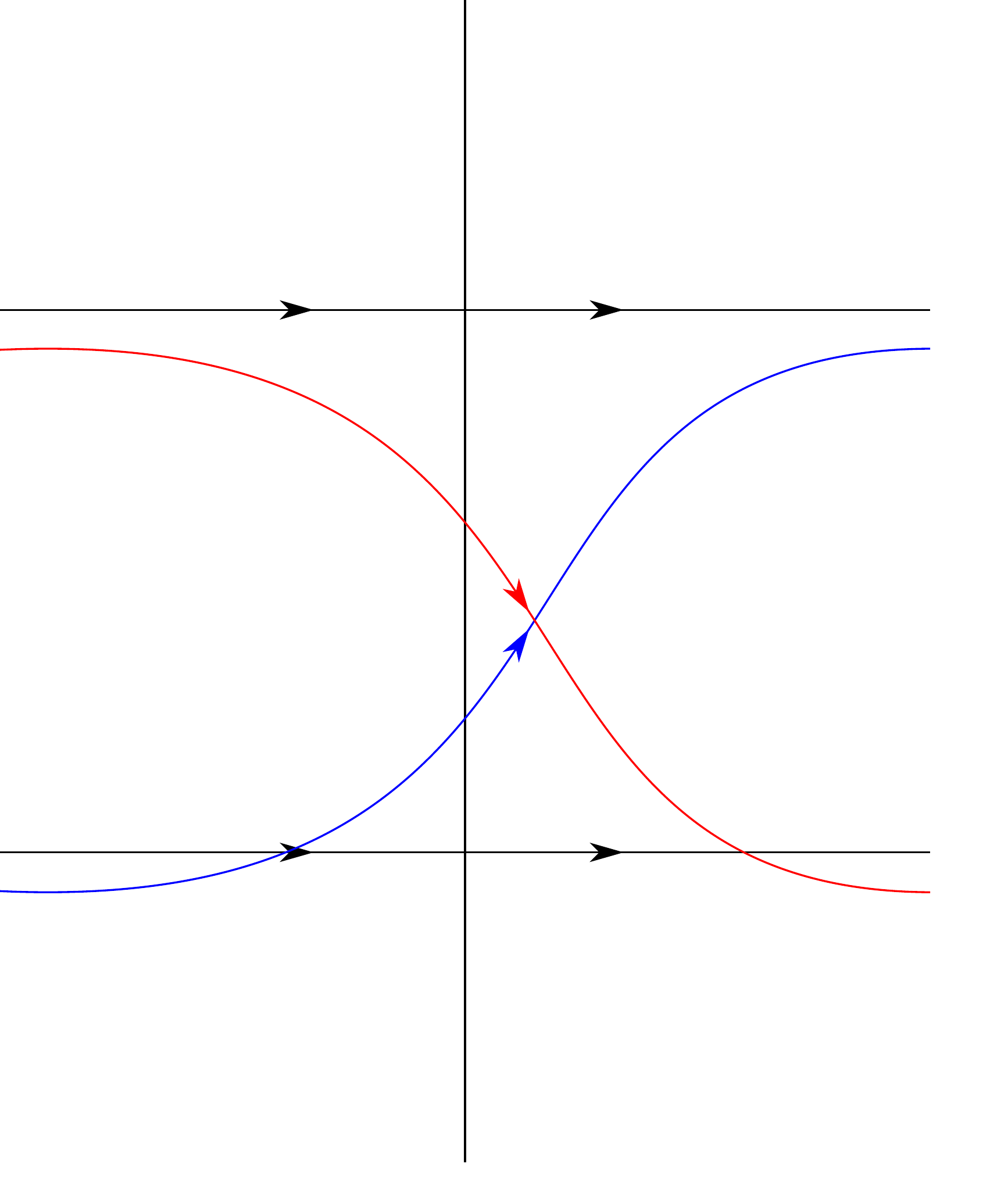 & 
\def\svgwidth{105pt}
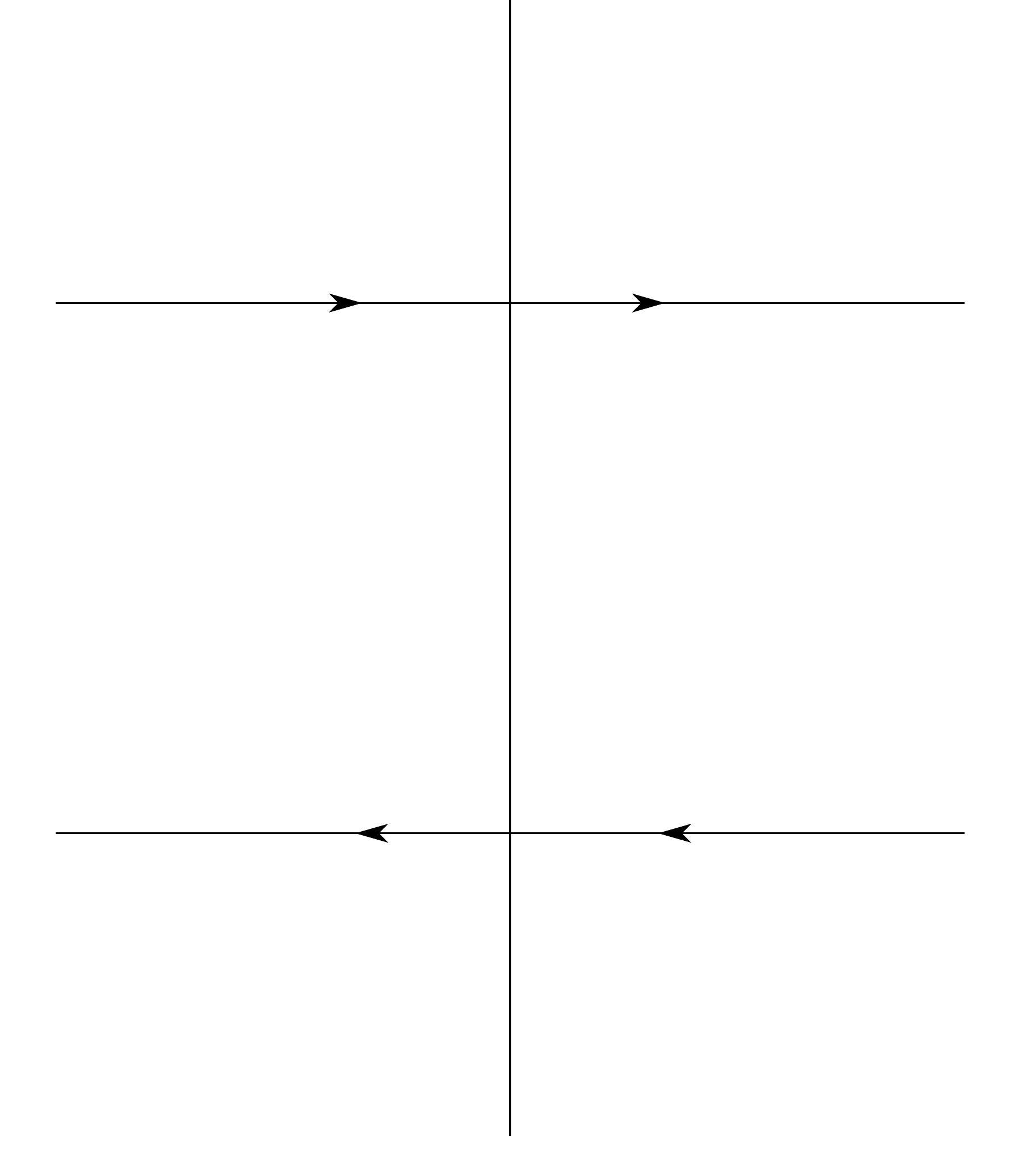 \\
\end{tabular}
\caption{}
\label{surgery}
\end{figure}

If the tangent vectors to $a$ at these intersection points are parallel along $b$ then perform surgery on $a$ with an arc of $b$ to define two oriented curves $c_1$ and $c_2$ as on the left of Figure \ref{surgery}. We have $[c_1]+[c_2]=[a]$ and hence since $[a]\neq 0$, we have $[c_1]\neq 0$ or $[c_2]\neq 0$. Hence there is a nonseparating curve $c$ with $i(c,a)=1$ and $i(c,b)\leq i(a,b)-1$. By induction we have \[d(a,b)\leq d(a,c)+d(c,b)\leq 2+2i(c,b)+1\leq 2+2i(a,b)-2+1=2i(a,b)+1.\]

If on the other hand the tangent vectors to $a$ at these intersection points are not parallel along $b$ then perform surgery on $a$ with an arc of $b$ to define two oriented curves $c_1$ and $c_2$ as on the right of Figure \ref{surgery}. Again we have $[c_1]+[c_2]=[a]$. Hence there is some nonseparating curve $c$ with $i(a,c)=0$ and $i(c,b)\leq i(a,b)-2$. Therefore we have \[d(a,b)\leq d(a,c)+d(c,b)\leq 1+2i(c,b)+1\leq 1+2i(a,b)-4+1=2i(a,b)-2.\]\end{proof}

\subsection{Proof of Theorem 1.1}
\label{proof}

\begin{figure}[h]
\begin{tabular}{c c}
\def\svgwidth{150pt}
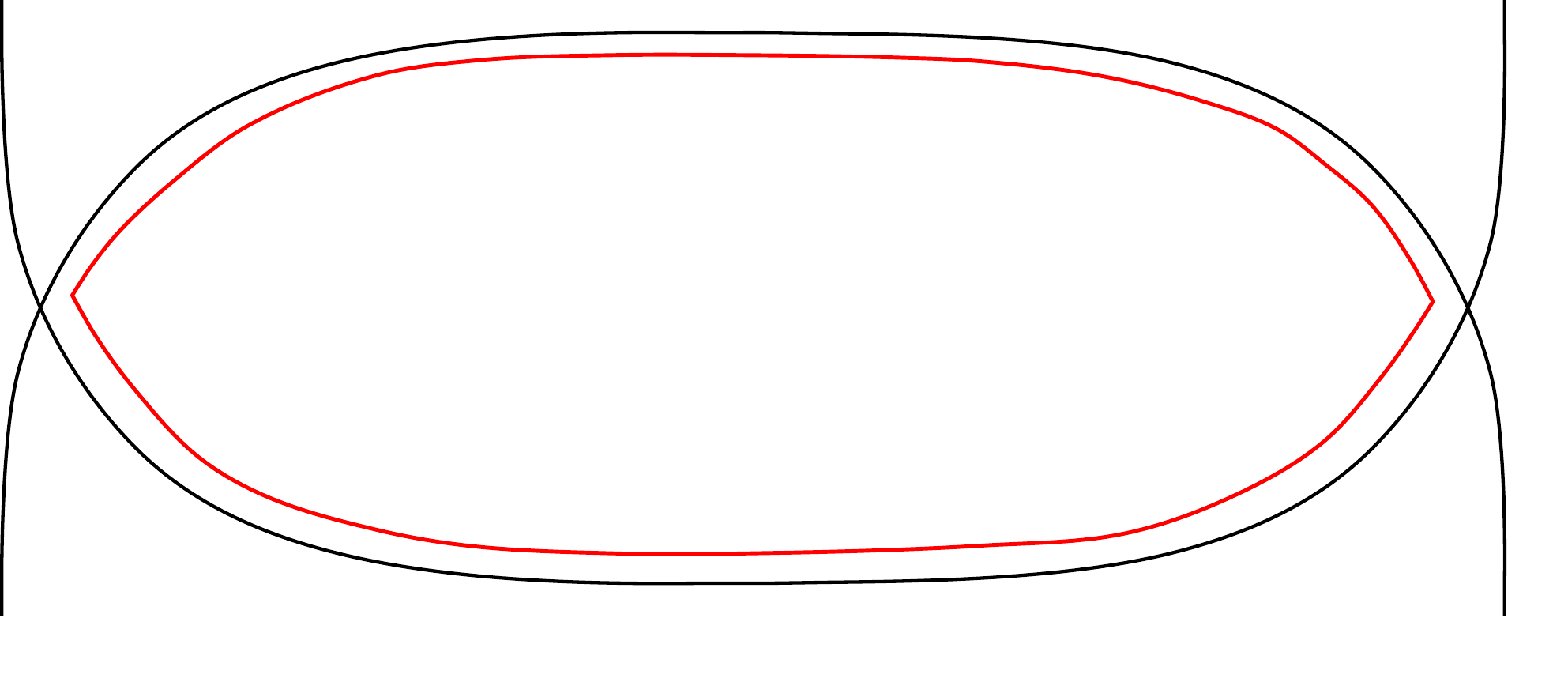 & 
\def\svgwidth{150pt}
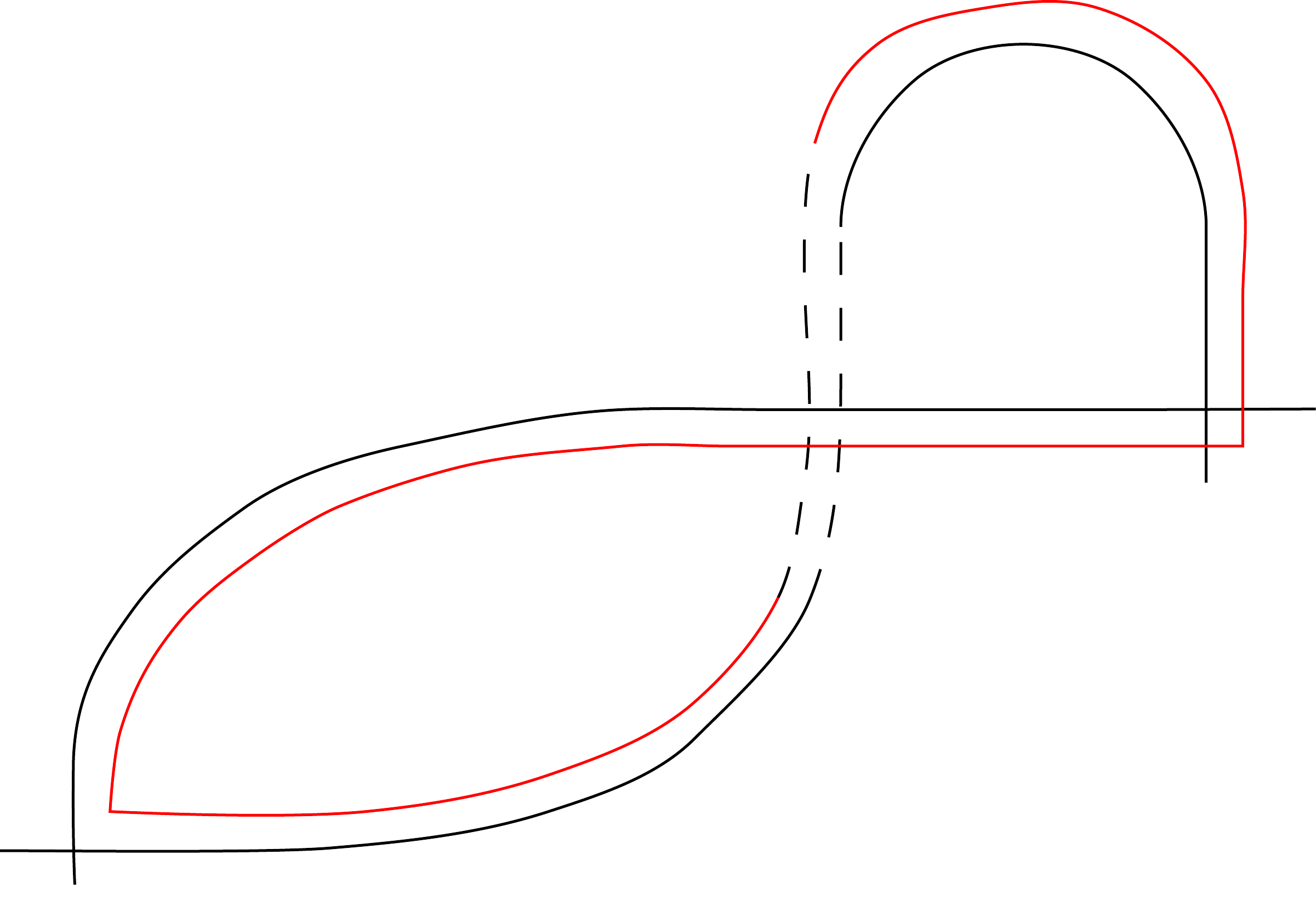 \\
\end{tabular}
\caption{The red curve $c$ is a bicorn between $a$ and $b$. Dotted lines are drawn to indicate that two curves do not intersect at a point where one of them is dotted.}
\label{bicorns}
\end{figure}

We begin the proof of hyperbolicity of $\NS(S)$. Realize all simple closed curves on $S$ mutually in minimal position (for instance by choosing a hyperbolic metric on $S$ and taking the geodesic representative of all simple closed curves). For $a,b\in \NS(S)^0$ we define a \textit{bicorn curve} between $a$ and $b$ to be a union $\alpha\cup \beta$ of an arc $\alpha$ of $a$ and an arc $\beta$ of $b$ intersecting at their endpoints $p$ and $q$ and nowhere in their interiors. We call $p$ and $q$ the \textit{corners} of the bicorn. Possible bicorns are pictured in Figure \ref{bicorns}. We also consider $a$ and $b$ themselves to be bicorns between $a$ and $b$. In the case of $a$, for instance, we set $\alpha$ to be $a$ itself and $\beta$ to be a point of $a\cap b$.

Many bicorns between $a$ and $b$ may be separating (even possibly \textit{peripheral}, i.e. bounding a once-punctured disk). Thus we define $A(a,b)$ to be the full subgraph of $\NS(S)$ spanned by \textit{nonseparating bicorn curves} between $a$ and $b$. We use the graphs $A(a,b)$ and the following criterion of Masur-Schleimer and Bowditch to prove that $\NS(S)$ is hyperbolic with hyperbolicity constant independent of $S$.

\begin{prop}[\cite{ms} Theorem 3.15, \cite{bow} Proposition 3.1]
Let $X$ be a graph and suppose that to each pair of vertices $a,b\in X^0$ we have associated a connected subgraph $A(a,b)$ containing $a$ and $b$. Suppose also that there exists $D>0$ such that

\begin{enumerate}[(i)]
\item if $a,b\in X^0$ are joined by an edge then the diameter of $A(a,b)$ is at most $D$,
\item for all $a,b,c\in X^0$, we have that $A(a,c)$ is contained in the $D$-neighborhood of $A(a,b)\cup A(b,c)$.
\end{enumerate}

Then $X$ is hyperbolic with hyperbolicity constant depending only on $D$.
\end{prop}

\begin{claim}
If $a,b\in \NS(S)^0$ are joined by an edge then $A(a,b)$ has diameter at most 2.
\end{claim}

\begin{proof}
If $i(a,b)\leq 1$ then $A(a,b)$ consists of $a,b$ and the edge between them. Hence $A(a,b)$ has diameter one in this case.

\begin{figure}[h]
\begin{tabular}{c c}
\def\svgwidth{125pt}
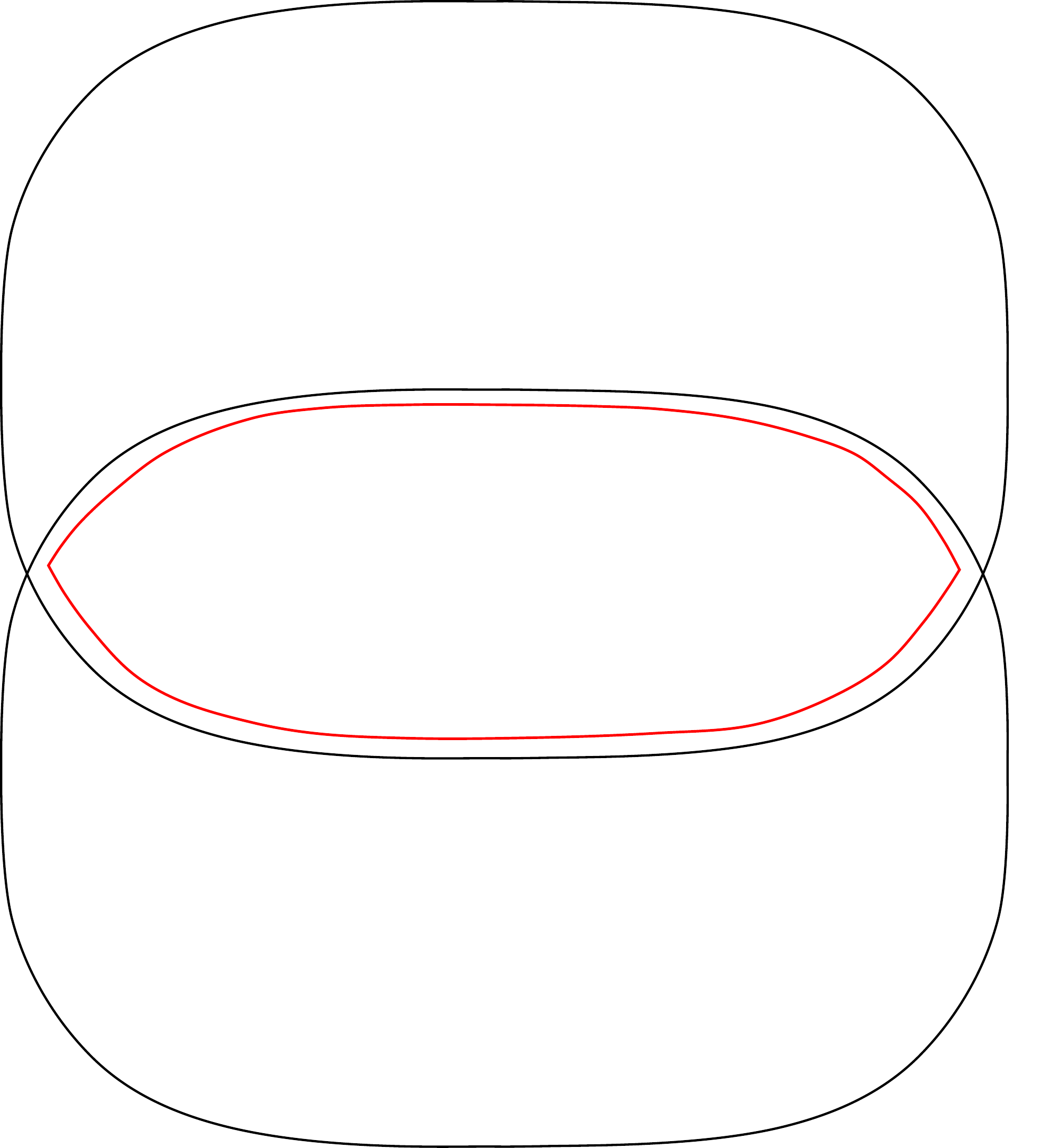 & 
\def\svgwidth{200pt}
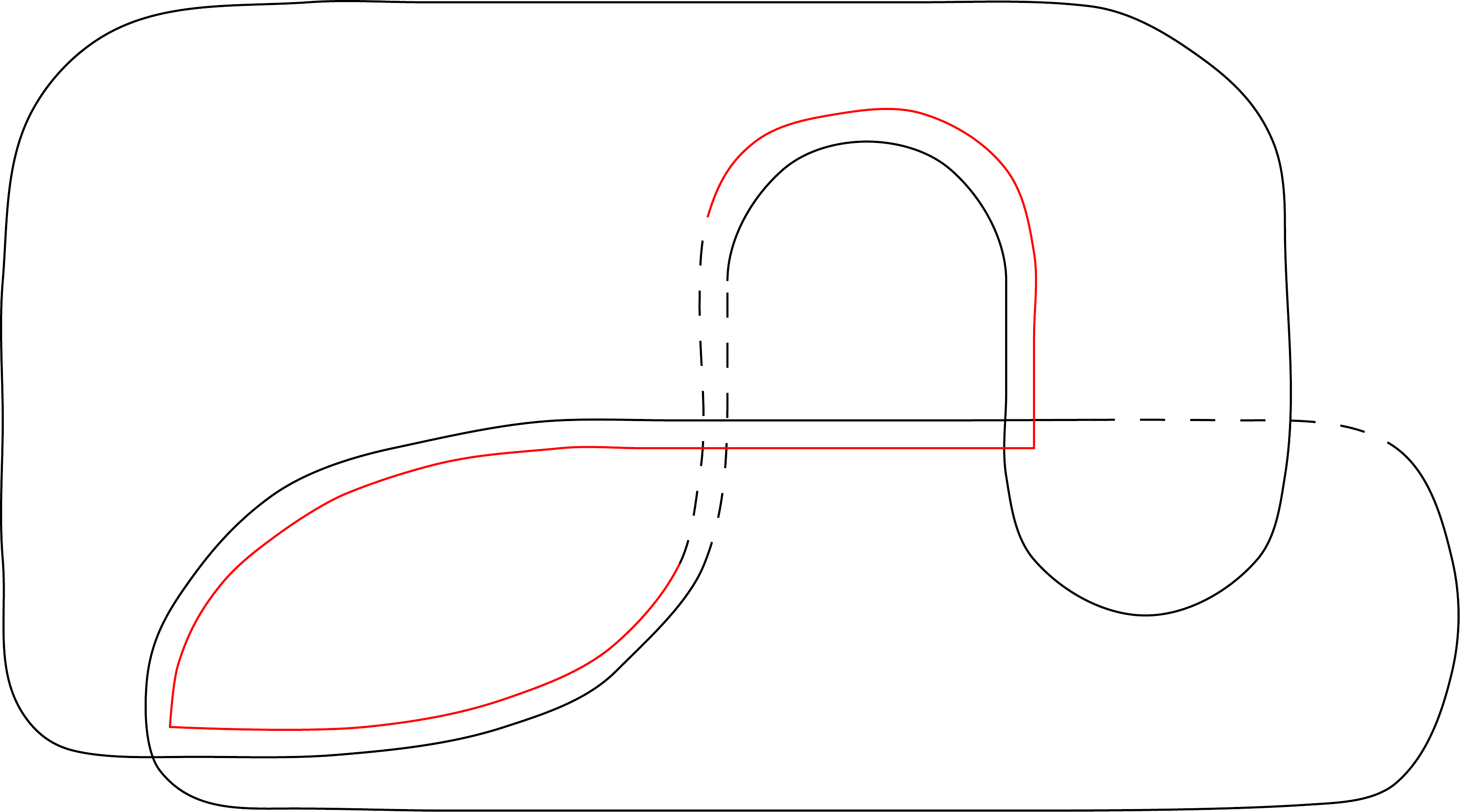 \\
\end{tabular}
\caption{If $i(a,b)=2$ then $a$ and $b$ have one of the two intersection patterns pictured above. Hence a bicorn $c$ intersects $a$ at most once.}
\label{small}
\end{figure}

If $i(a,b)=2$ then each vertex $c$ of $A(a,b)$ has $i(a,c)\leq 1$. See Figure \ref{small}.\end{proof}

\begin{claim}
$A(a,b)$ is connected.
\end{claim}

\begin{proof}
Note that if $i(a,b)\leq 1$ then $A(a,b)$ consists of the full subgraph of $\NS(S)$ spanned by $a$ and $b$, which is a single edge (therefore connected). Hence we may assume $i(a,b)\geq 2$.

There is a partial order defined on vertices of $A(a,b)$ as follows: $c<c'$ if the $b$-arc of $c'$ properly contains the $b$-arc of $c$. Technically this partial order is defined on representatives of bicorn curves as unions of a $b$-arc and an $a$-arc, however this will not matter for the purposes of this proof. We show that given any $c\in A(a,b)$, if $c\neq b$ then we may find $c'\in A(a,b)$ with $c<c'$ and $i(c,c')\leq 2$ so that $c$ and $c'$ are adjacent in $\NS(S)$. Since $a$ and $b$ intersect finitely many times, this will complete the proof.

First consider the case that $c=a$. In this case, choose a point $x$ of $a\cap b$. Fix an orientation of $b$ and let $y$ be the first point of $a\cap b$ after $x$ along $b$. We obtain a subarc $\beta$ of $b$ with endpoints $x$ and $y$, which intersects $a$ nowhere in its interior. The points $x$ and $y$ bound two subarcs $\alpha$ and $\alpha'$ of $a$ with disjoint interiors such that $a=\alpha\cup \alpha'$. Note that $c'_1=\alpha\cup \beta$ and $c'_2=\alpha'\cup \beta$ are bicorns and that $i(c'_i,a)\leq 1$ for $i=1,2$. Moreover, we have $[a]=[c_1']+[c_2']$ when the curves are oriented appropriately. Hence at least one of $c_1'$ or $c_2'$ is nonseparating. Choose $c'$ to be such a nonseparating bicorn and note that we have $a=c<c'$.

Otherwise we have a nonseparating bicorn $c=\alpha\cup\beta$ where $\alpha$ is a nontrivial subarc of $a$ and $\beta$ is a nontrivial subarc of $b$. In this case extend $\beta$ past one of its endpoints along $b$ until it intersects the interior of $\alpha$ at a point $z$. This gives a subarc $\beta'$ of $b$ properly containing $\beta$ and intersecting $\alpha$ exactly once in its interior. In the special case that the extended $b$-arc never intersects the interior of $\alpha$, then one may check that we have $i(c,b)\leq 1$ so that we may take $c'=b$.

Let $x,y$ be the endpoints of $\beta$ so that $x,y,$ and $z$ occur in that order along $\beta'$. The points $x$ and $z$ bound a subarc $\alpha'$ of $\alpha$ such that $\alpha'\cup\beta'=c'$ is a bicorn of $a$ with $b$. There are two cases to consider.

\begin{figure}[h]
\begin{tabular}{c c}
\def\svgwidth{200pt}
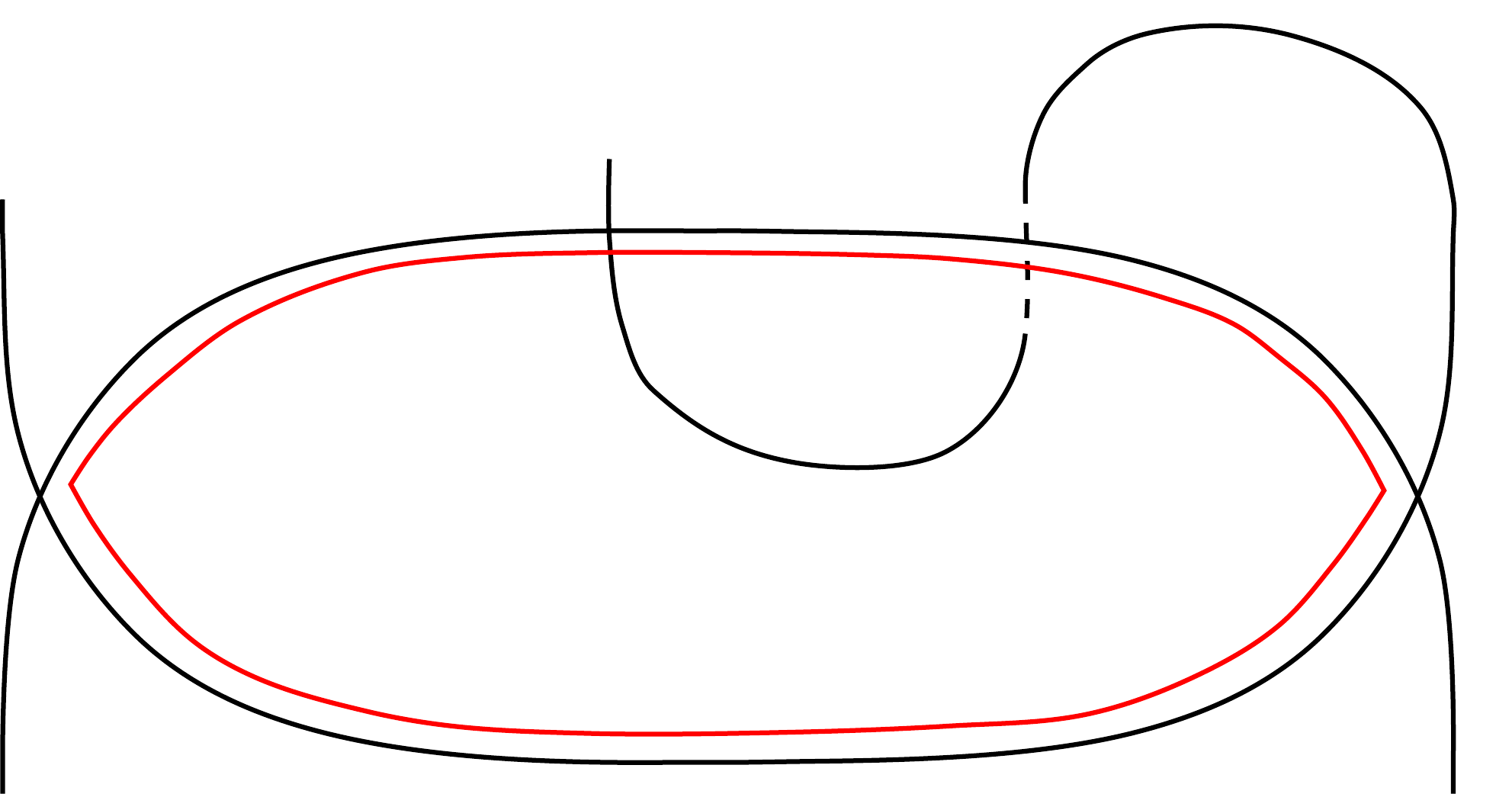 &
\def\svgwidth{200pt}
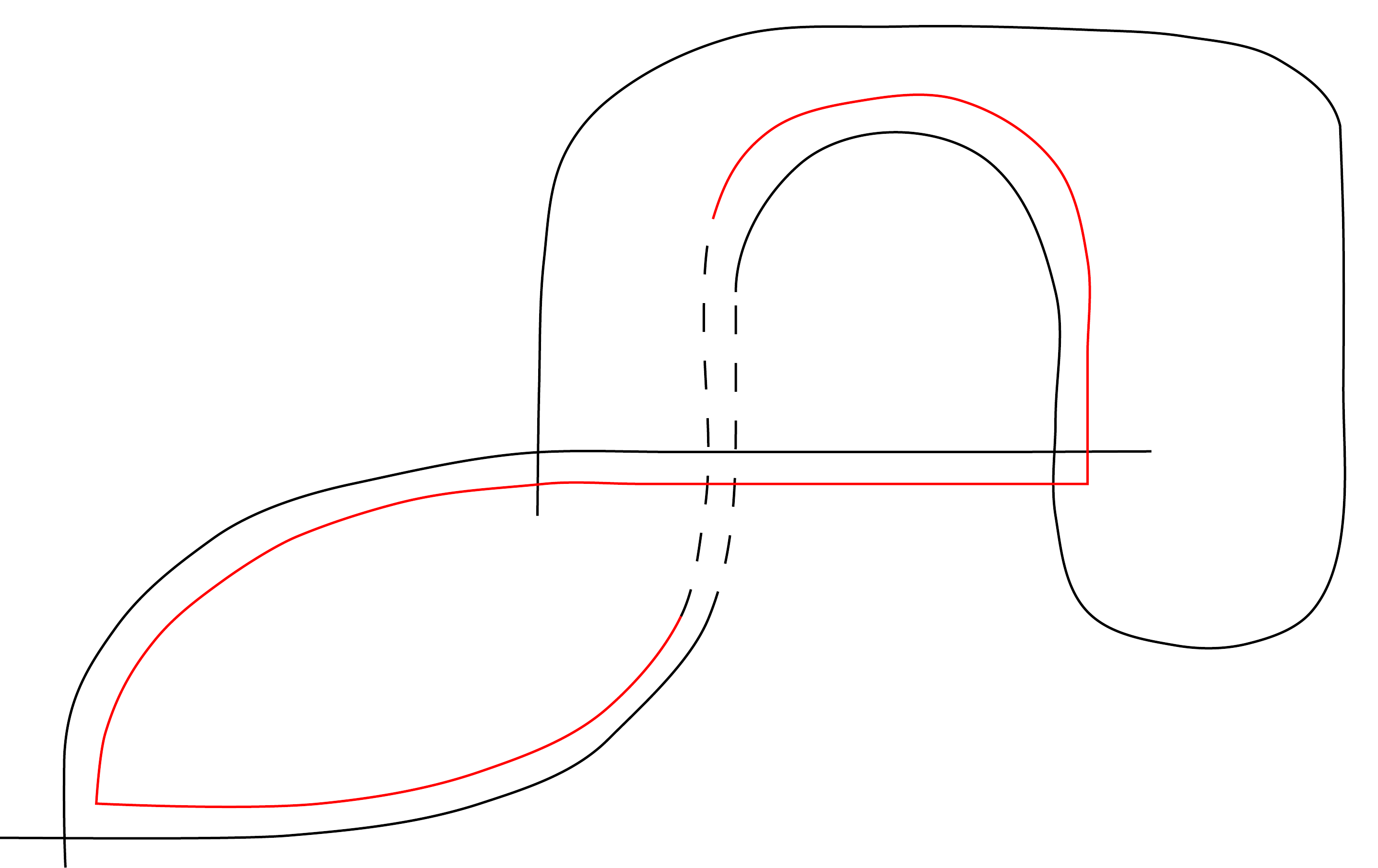 \\
\end{tabular}
\caption{If the signs of intersection of $a$ with $b$ at $y$ and $z$ are the same, then $i(c,c')=1$ and $c'$ is a nonseparating bicorn.}
\label{same_sign}
\end{figure}

First, if the signs of intersection of $a$ with $b$ at $y$ and $z$ are the same, then $i(c,c')=1$. Hence $c'$ is nonseparating so $c'\in A(a,b)$, $c$ and $c'$ are adjacent, and $c<c'$. See Figure \ref{same_sign}.

\begin{figure}[h]
\begin{tabular}{c c}
\def\svgwidth{190pt}
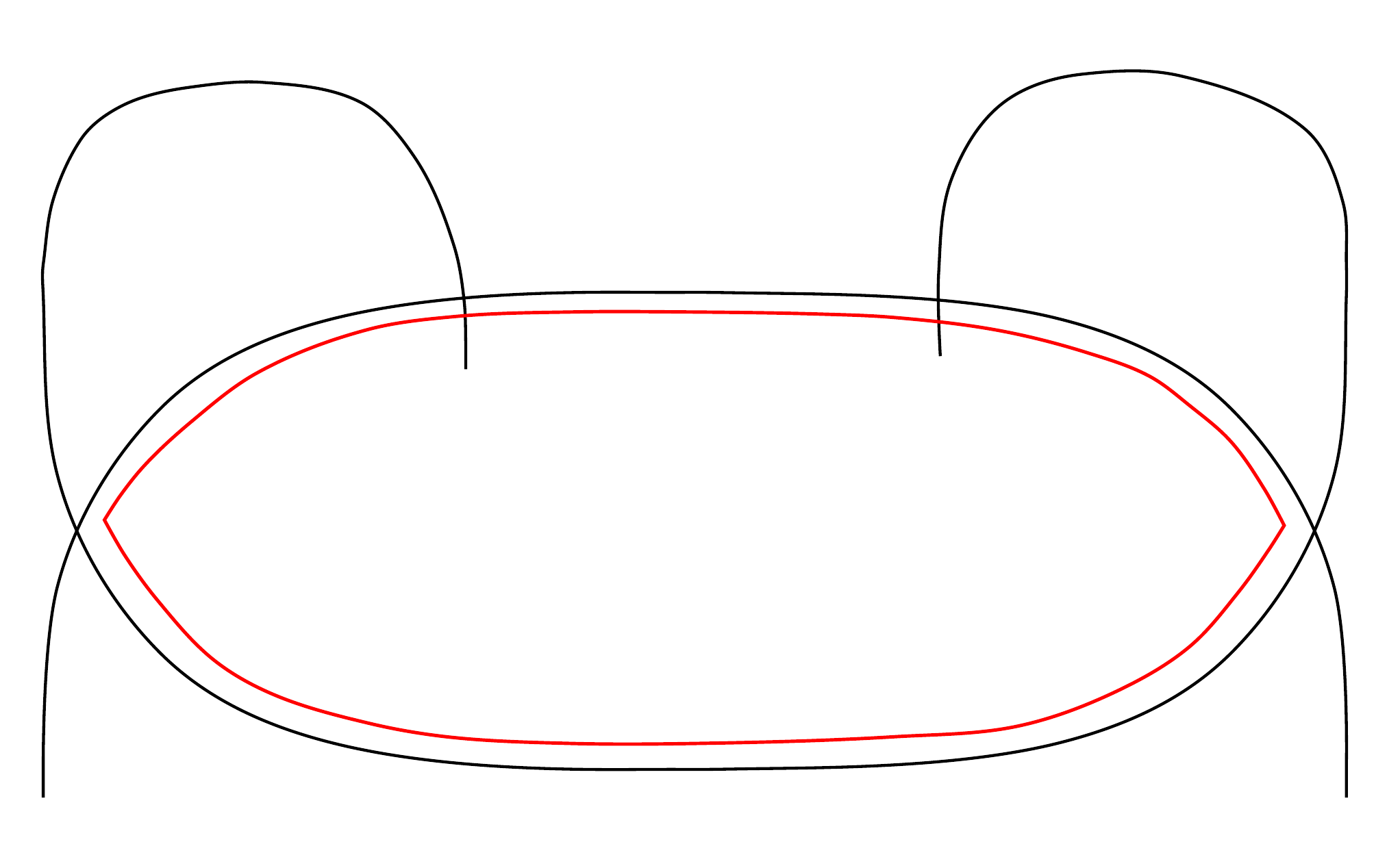 & 
\def\svgwidth{190pt}
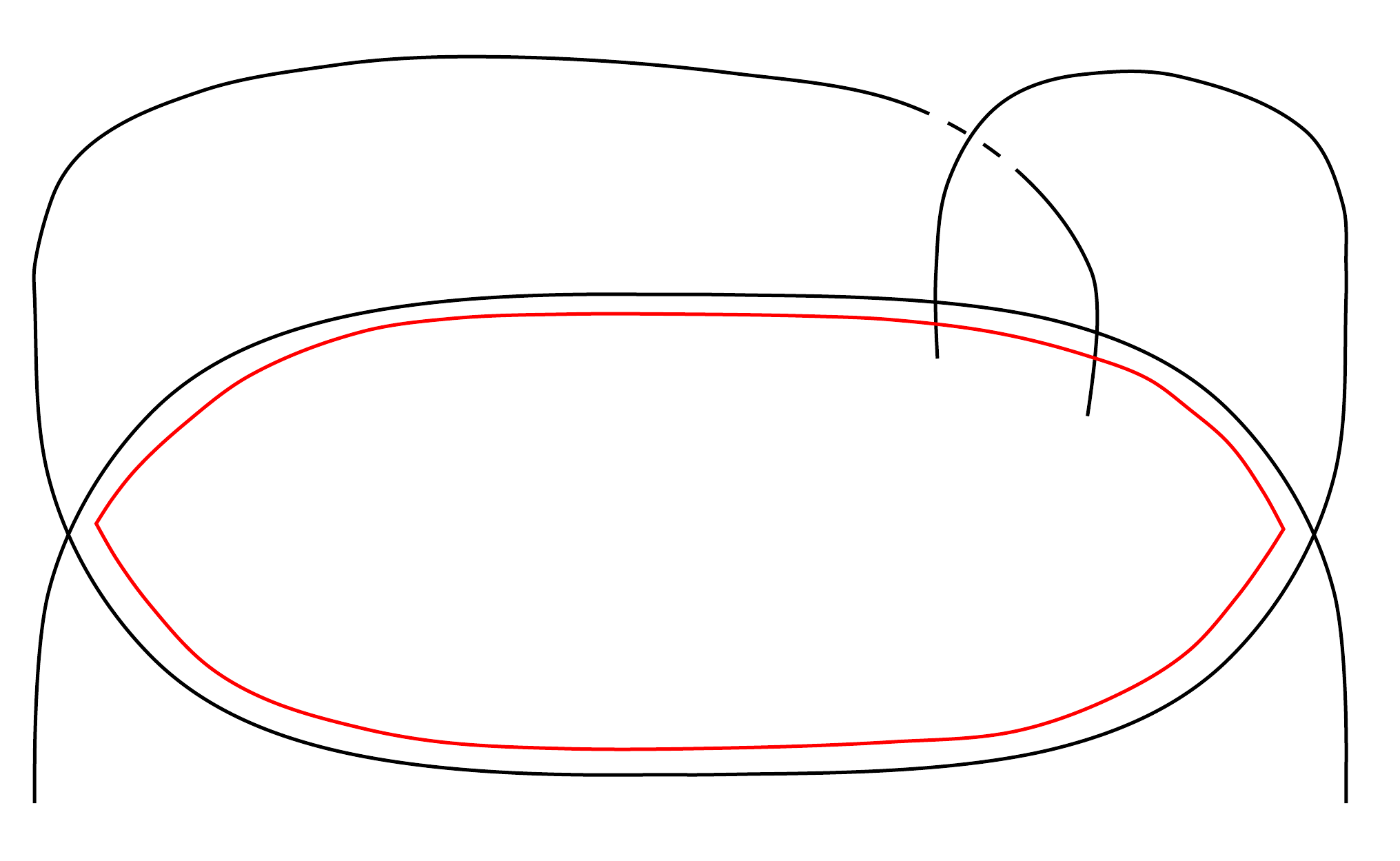 \\
\def\svgwidth{230pt}
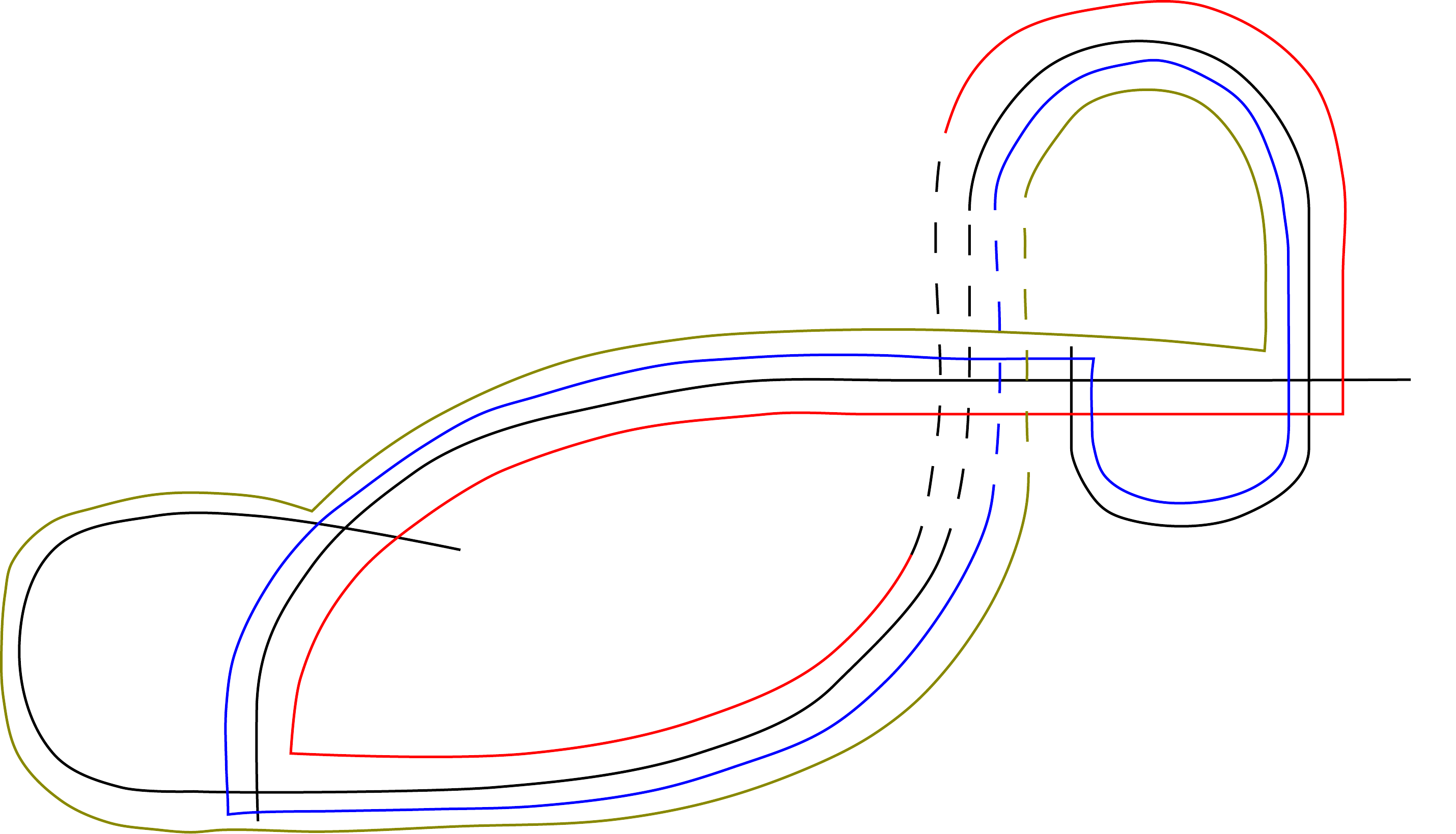 &
\def\svgwidth{230pt}
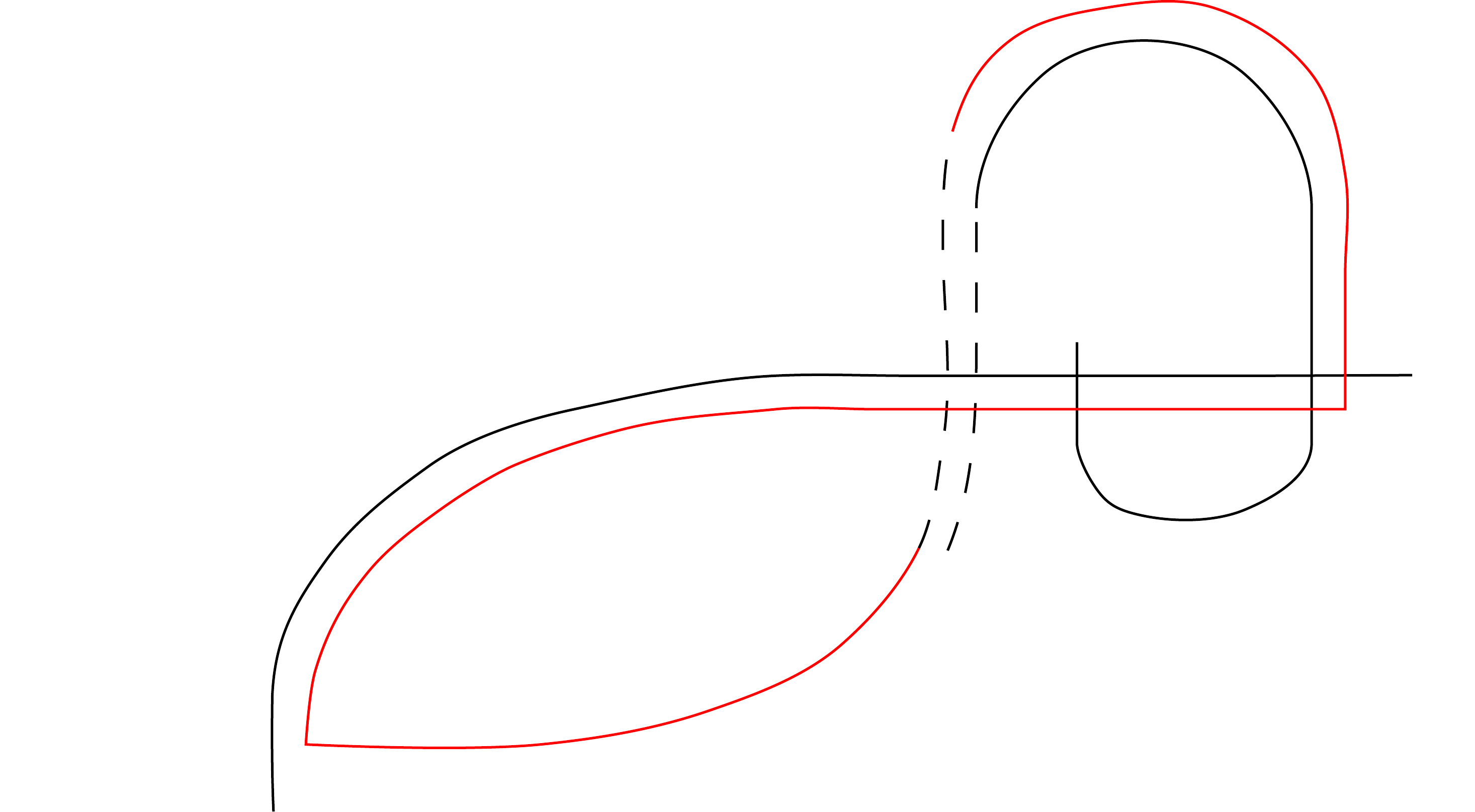 \\
\end{tabular}
\caption{}
\label{both}
\end{figure}

Therefore without loss of generality we suppose that the signs of intersection differ regardless of the direction in which we extend $\beta$. In this case we extend $\beta$ past \textit{both} of its endpoints until it intersects $\intr(\alpha)$ for the first time on each side. We thus find two bicorns of $a$ with $b$, $c_1$ and $c_2$, both with $b$-arc properly containing $\beta$. The bicorn $c_1$ has $x$ as a corner and the bicorn $c_2$ has $y$ as a corner. If $c_1$ and $c_2$ intersect $\intr(\alpha)$ in the same point then one may check that we have $i(c,b)\leq 2$ and therefore we may take $c'=b$. Otherwise we have one of the four intersection patterns pictured in Figure \ref{both}. We deal with the cases on the bottom of Figure \ref{both} explicitly. The other two cases are entirely analogous.

\begin{figure}[h]
\begin{tabular}{c}
\def\svgwidth{400pt}
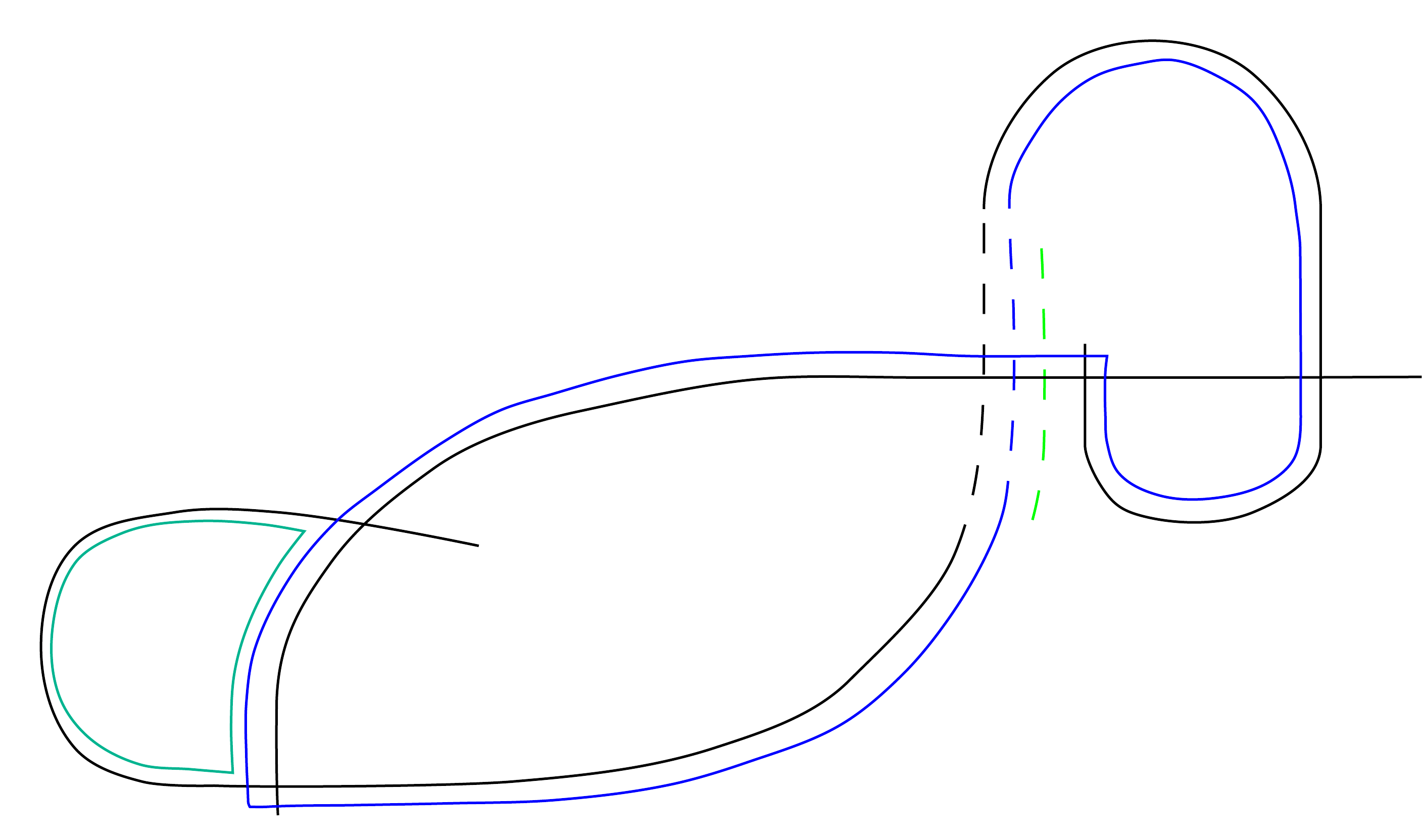 \\
\def\svgwidth{400pt}
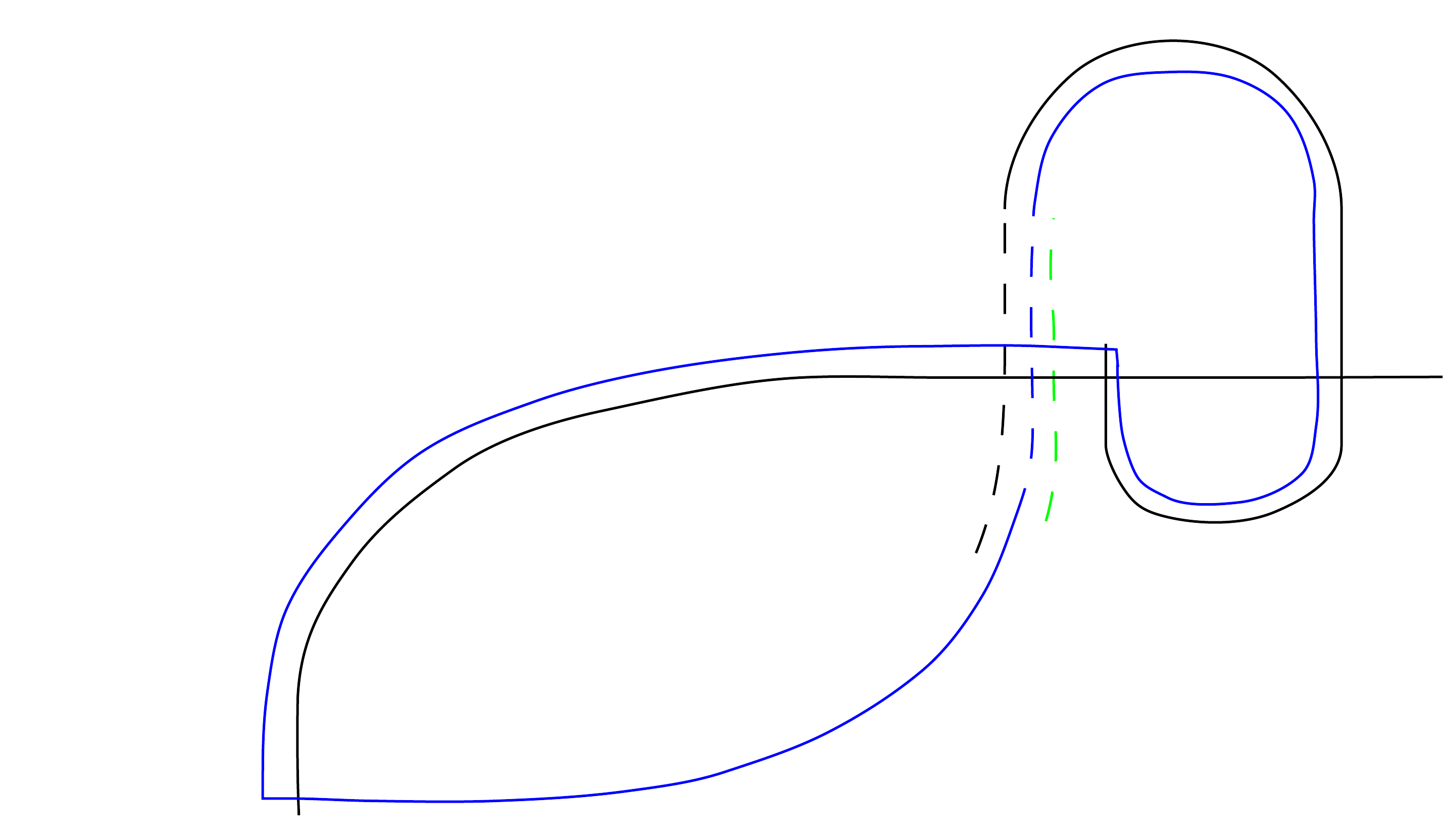 \\
\end{tabular}
\caption{If the bicorns $c_1$ and $c_2$ are both separating, then the bicorn $c'$ is nonseparating and we have $i(c,c')\leq2$.}
\label{nonsep}
\end{figure}

In the cases on the bottom of Figure \ref{both}, consider the two bicorns $c_1$ and $c_2$ defined as in the last paragraph. If $c_i$ is nonseparating for some $i$ then we may take $c'=c_i$ and therefore $c'$ is a nonseparating bicorn with $c<c'$ and $i(c,c')\leq 1$. If both $c_1$ and $c_2$ are separating, we note the following. There exists a bicorn $e_2$ between $a$ and $b$ such that $[c_2]+[e_2]=[c]$ when all three are oriented appropriately.  See Figure \ref{nonsep}. Since $c_2$ is separating and $c$ is nonseparating, this implies that $e_2$ is nonseparating. Therefore we find a third bicorn $c'$ with $[c']=[c_1]+[e_2]\neq 0$ and such that $i(c,c')\leq2$. We have $c<c'$.\end{proof}

\begin{claim}
There exists a universal constant $D>0$ with the following property. Let $a,b,d\in \NS(S)^0$. Then $A(a,b)$ is contained in the $D$-neighborhood of $A(a,d)\cup A(d,b)$.
\label{thin}
\end{claim}

\begin{proof}
Let $c\in A(a,b)$ with $a$-arc $\alpha$ and $b$-arc $\beta$. If $d$ happens to intersect both $\alpha$ and $\beta$ at most once, then we have $i(c,d)\leq 2$ and hence $c$ is distance one from $d\in A(a,d)$. Otherwise $d$ intersects one of the sides of $c$ at least twice and we may suppose without loss of generality that it intersects $\beta$ at least twice. Orient $d$ and enumerate the intersections of $d$ with $\beta$, $x_1,x_2,\ldots,x_m$ in the order they appear along $d$.

For each $i=1,\ldots,m$, there is a unique subarc $\delta_i$ of $d$ oriented from $x_i$ to $x_{i+1}$ such that the resulting orientation of $\delta_i$ agrees with the orientation of $d$ (indices being taken modulo $m$). Moreover, $x_i$ and $x_{i+1}$ bound a unique subarc $\beta_i$ of $\beta$. The union $\beta_i\cup \delta_i$ is a bicorn $c_i'$ of $b$ with $d$. Moreover, if we give $c_i'$ the orientation induced by $\delta_i$, we have that $[d]$ is the sum of the $[c_i']$. To see this, note that \[\sum_{i=1}^m[c_i'] = \left[ \sum_{i=1}^m \delta_i\right] + \left[\sum_{i=1}^m \beta_i\right]=[d]+\left[\sum_{i=1}^m \beta_i\right]\] where we orient $\delta_i$ from $x_i$ to $x_{i+1}$ and $\beta_i$ from $x_{i+1}$ to $x_i$. Here the sum $\sum \beta_i$ is a singular 1-cycle contained in the image of $\beta$. Hence this cycle is nullhomologous and the result follows. As a result, one of these bicorns $c_i'$ is nonseparating. Setting $c'=c_i'$, we obtain a nonseparating bicorn $c'\in A(b,d)$ with $b$-arc $\beta'\subset \beta$.

\begin{figure}[h]
\centering
\def\svgwidth{\textwidth}
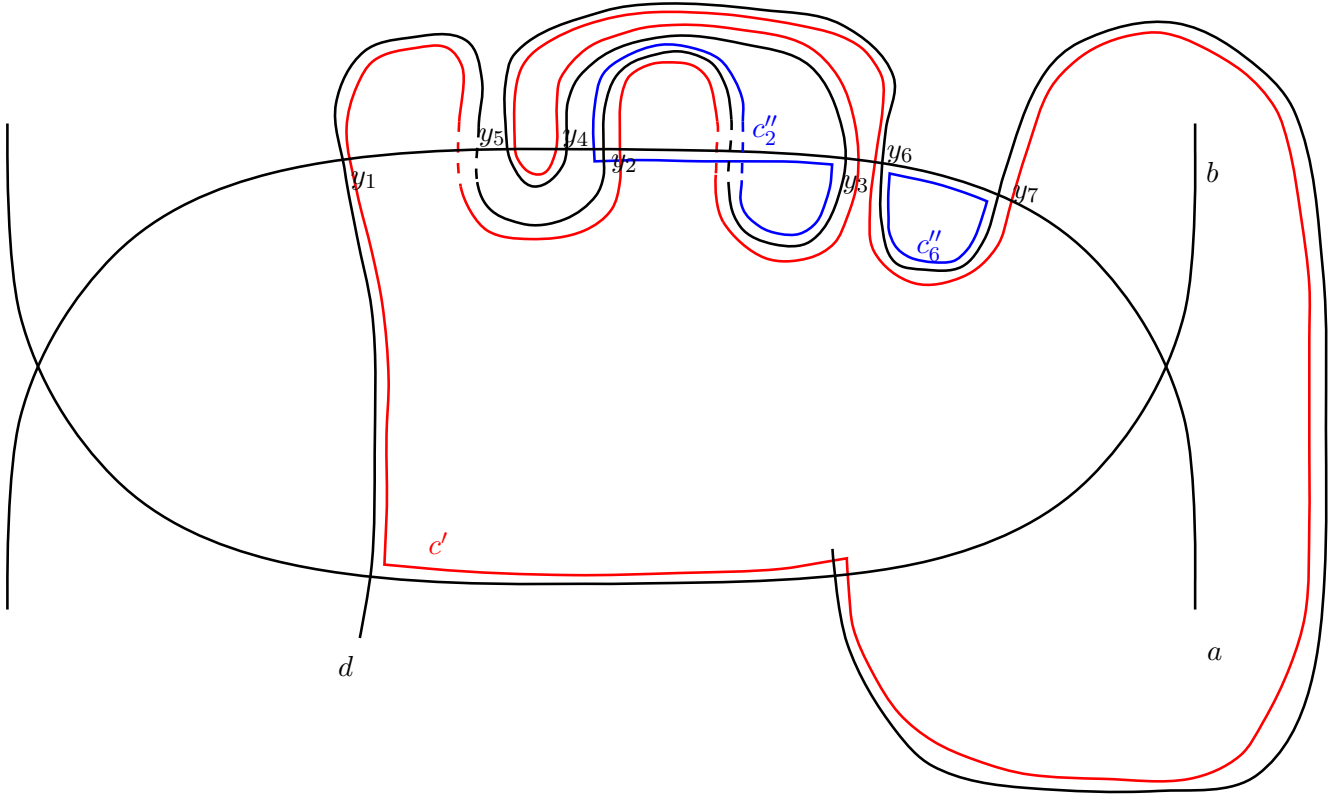
\caption{A possible picture of the bicorn $c'\in A(a,d)$. Consecutive intersections of $c'$ with $\alpha$ give bicorns $c_i''$ of $c'$ with $a$ and hence of $d$ with $a$. Two of these are shown above.}
\label{b_d_bicorn}
\end{figure}

If $c'$ happens not to intersect $\alpha$ then we have $i(c',c)\leq 1$ and $c$ is distance one from $c'\in A(b,d)$. Otherwise, starting from one of the two endpoints of $\beta'$, enumerate the points of intersection of $c'$ with $\alpha$, $y_1,y_2,\ldots,y_n$ in the order that they appear along $c'$. For $i=1,\ldots,n-1$, $y_i$ and $y_{i+1}$ bound a unique subarc $\alpha_i$ of $\alpha$. Moreover there is a unique subarc $\gamma_i$ of $c'$ from $y_i$ to $y_{i+1}$ not containing the $b$-arc $\beta'$ of $c'$. We obtain a bicorn  of $a$ with $d$, $c_i''=\alpha_i\cup\gamma_i$ for $i=1,\ldots,n-1$. See Figure \ref{b_d_bicorn}. Moreover, $c_i''$ satisfies $i(c,c_i'')\leq 1$. Hence if some $c_i''$ is nonseparating, we obtain $c''\in A(a,d)$ adjacent to $c$ and the proof is then complete.

So suppose that each $c_i''$ is separating. We claim in this case that $c$ is a uniformly bounded distance from $c'\in A(b,d)$. Orient $\alpha$. From the fact that each $c_i''$ is separating, we obtain the following:
\begin{enumerate}[(i)]
\item For each $i=1,\ldots,n-1$, the $d$-arc $\gamma_i$ either joins the left side of $\alpha$ to the left side of $\alpha$ or the right side of $\alpha$ to the right side of $\alpha$. For otherwise $c_i''$ would intersect $c$ exactly once and thus would be nonseparating.
\item For each $1\leq i< j\leq n-1$, if the $d$-arcs of $c_i''$ and $c_j''$, namely $\gamma_i$ and $\gamma_j$, both join the left side of $\alpha$ to the left side of $\alpha$, then the $a$-arcs of $c_i''$ and $c_j''$, namely, $\alpha_i$ and $\alpha_j$, are either nested or disjoint. For suppose that $\gamma_i$ has endpoints $x,z$ and $\gamma_j$ has endpoints $y,w$ with $x<y<z<w$ in the orientation of $\alpha$. Then the bicorns $c_i''$ and $c_j''$ have $i(c_i'',c_j'')=1$. This contradicts that $c_i''$ and $c_j''$ are separating.
\end{enumerate} 

\begin{figure}[h]
\centering
\def\svgwidth{\textwidth}
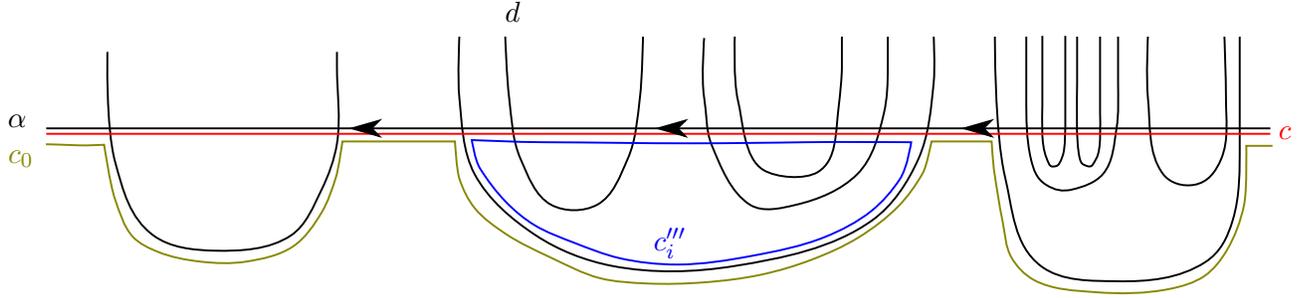
\caption{A picture of the modified curve $c_0$ along with one of the separating bicorns $c_i'''$. These bicorns are chosen to have the property that their $\alpha$-arcs are maximal with respect to inclusion.}
\label{modified}
\end{figure}

Now enumerate all of the $d$-arcs $\gamma_j$ which join the left side of $\alpha$ to the left side of $\alpha$ and such that the corresponding $a$-arcs $\alpha_j$ are maximal with respect to inclusion: $\gamma_1',\ldots,\gamma_r'$. Let $\alpha_1',\ldots,\alpha_r'$ be the corresponding $a$-arcs and $c_1'''=\gamma_1'\cup\alpha_1',\ldots,c_r'''=\gamma_r'\cup\alpha_r'$ be the resulting bicorns of $a$ with $d$. For each $i=1,\ldots,r$ replace the subarc $\alpha_i'$ of $c$ with $\gamma_i'$. Call the resulting curve $c_0$. See Figure \ref{modified}. From the construction one finds that $c_0$ is simple. As the homology class $[c_0]$ is the sum of $[c]$ with the classes $[c_1''']=\ldots=[c_r''']=0$, we have $[c_0]\neq 0$ so that $c_0$ is nonseparating. We finally claim that $i(c,c_0)=0$ and that $i(c_0,c')\leq 3$. This will finish the proof.

The new curve $c_0$ consists of the $b$-arc $\beta$ along with several subarcs of $\alpha$ and of $d$. The $d$-arcs of $c_0$ join the left side of $\alpha$ to the left side of $\alpha$ and intersect $\alpha$ nowhere in their interiors. By homotoping the $\alpha$-arcs and the $\beta$-arc of $c_0$ slightly to the left of $c$ we see that $i(c,c_0)=0$ (see Figure \ref{modified}). To see that $i(c_0,c')\leq 3$, first homotope the $b$-arc $\beta'$ of $c'$ slightly off of $\beta$ so as to create no bigons between $c_0$ and $c'$. Then, homotope the $d$-arcs of $c_0$ slightly off of themselves within a regular neighborhood of $d$, moving their endpoints slightly into the interiors of the adjacent $\alpha$-arcs of $c_0$. At this point, the homotoped $b$-arc of $c'$ accounts for at most one point of intersection of $c'$ with $c_0$. Any other intersections of $c_0$ with $c'$ must occur between the $d$-arc of $c'$ and one of the $\alpha$-subarcs of $c_0$. By definition, such a point of intersection is some $y_i$. However, each $y_i$ other than possibly $y_1$ or $y_n$ lies on both of the arcs $\gamma_i$ and $\gamma_{i-1}$ and one of these arcs joins the left side of $\alpha$ to the left side. In particular, if $i\neq 1,n$, then $y_i$ lies on one of the maximal $a$-arcs $\alpha'_j$, which has been removed to define $c_0$. So the only points of intersection of $c'$ with $c_0$, besides the one already mentioned, are $y_1$ and/or $y_n$. This proves that $i(c,c_0)\leq 3$, as claimed. See Figure \ref{c_c'_intersection} for an example of the possible intersection pattern of $c_0$ and $c'$.

\begin{figure}[h]
\centering
\def\svgwidth{350pt}
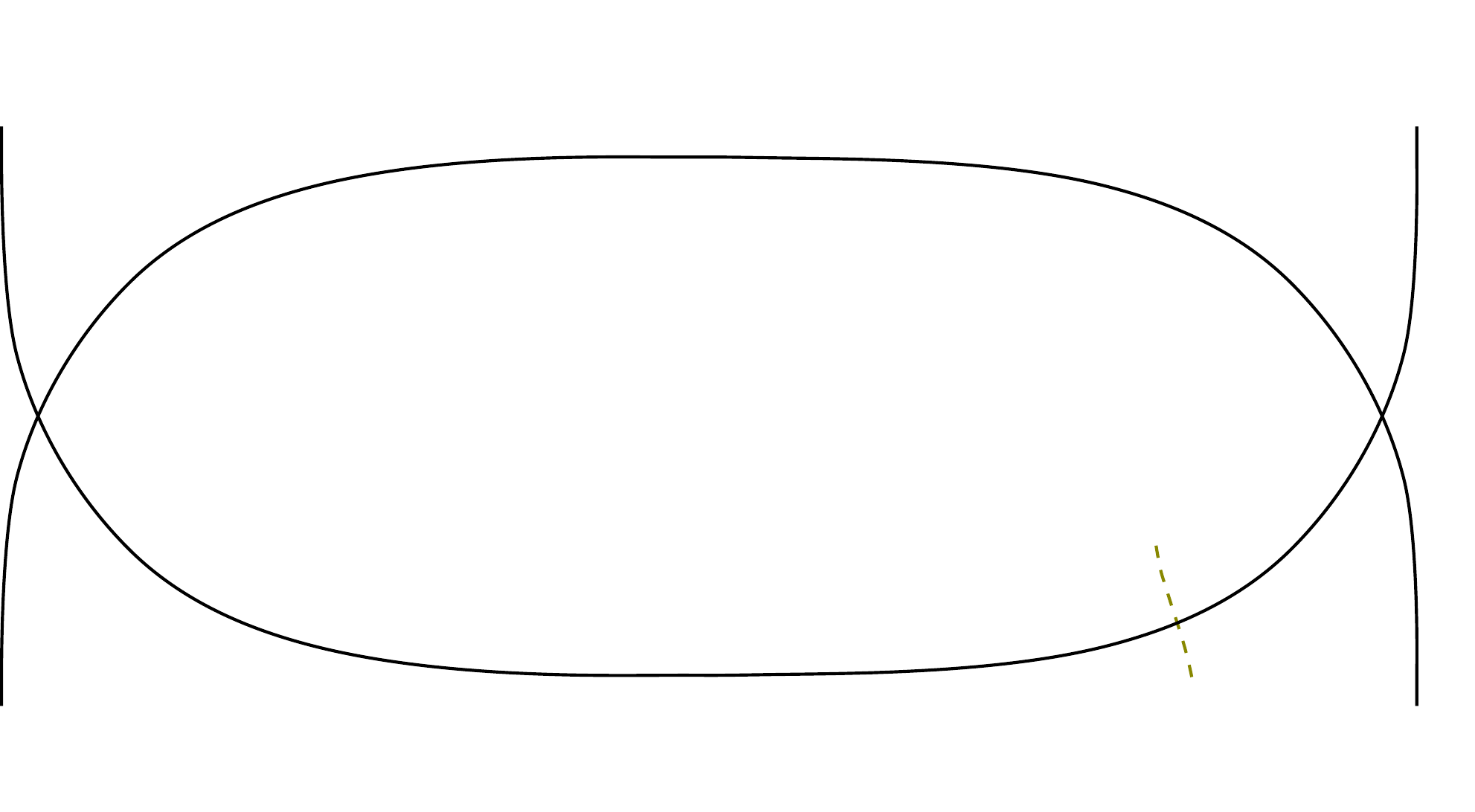
\caption{A picture of $c_0$ and $c'$ and their points of intersection. The bicorn $c$ is oriented counterclockwise in this picture.}
\label{c_c'_intersection}
\end{figure}\end{proof}

\section{The case of genus one}
\label{genus1}

If $S$ is a finite type oriented surface of genus zero then the graph $\NS'(S)$ defined earlier is empty. If $S$ has genus one then $\NS'(S)$ is non-empty but disconnected (to see this, note that if two distinct nonseparating curves on $S$ are disjoint then they are homotopic after capping off all the boundary components and punctures). In this case we redefine $\NS'(S)$ so that two vertices are joined by an edge if they represent curves with intersection number at most one.

So let $S$ be a finite type surface of genus one. It remains to prove Theorem 1.1 for this graph $\NS'(S)$. We have the following modification of Lemma \ref{surg}:

\begin{lem}
For $a,b\in\NS'(S)^0$, $d_{\NS'(S)}(a,b)\leq 2i(a,b)+1$.
\label{surg1}
\end{lem}

\begin{proof}
Again we prove the statement by induction on $i(a,b)$. In the base case $i(a,b)=1$ we have $d_{\NS'(S)}(a,b)=1$ by definition.

The rest of the proof goes through unmodified.\end{proof}

Once again we define $\NS(S)$ to have vertex set consisting of the free homotopy classes of simple nonseparating curves on $S$ with edges joining curves of intersection number at most 2. As before, Lemma \ref{surg1} proves that $\NS(S)$ and $\NS'(S)$ are quasi-isometric. Moreover, the proofs from section \ref{proof} go through unmodified, using Lemma \ref{surg1} in place of Lemma \ref{surg}.

\bibliographystyle{plain}
\bibliography{nsreferences}

\end{document}

%% file: same_orientation.pdf_tex
\begingroup%
  \makeatletter%
  \providecommand\color[2][]{%
    \errmessage{(Inkscape) Color is used for the text in Inkscape, but the package 'color.sty' is not loaded}%
    \renewcommand\color[2][]{}%
  }%
  \providecommand\transparent[1]{%
    \errmessage{(Inkscape) Transparency is used (non-zero) for the text in Inkscape, but the package 'transparent.sty' is not loaded}%
    \renewcommand\transparent[1]{}%
  }%
  \providecommand\rotatebox[2]{#2}%
  \ifx\svgwidth\undefined%
    \setlength{\unitlength}{831.93675604bp}%
    \ifx\svgscale\undefined%
      \relax%
    \else%
      \setlength{\unitlength}{\unitlength * \real{\svgscale}}%
    \fi%
  \else%
    \setlength{\unitlength}{\svgwidth}%
  \fi%
  \global\let\svgwidth\undefined%
  \global\let\svgscale\undefined%
  \makeatother%
  \begin{picture}(1,1.19422949)%
    \put(0,0){\includegraphics[width=\unitlength,page=1]{same_orientation.pdf}}%
    \put(0.94242294,0.34801139){\color[rgb]{0,0,0}\makebox(0,0)[lb]{\smash{$a$}}}%
    \put(0.94242294,0.88651376){\color[rgb]{0,0,0}\makebox(0,0)[lb]{\smash{$a$}}}%
    \put(0.46161722,0.00183127){\color[rgb]{0,0,0}\makebox(0,0)[lb]{\smash{$b$}}}%
    \put(0.94242294,0.27108247){\color[rgb]{0,0,0}\makebox(0,0)[lb]{\smash{$c_1$}}}%
    \put(0.94242294,0.80958485){\color[rgb]{0,0,0}\makebox(0,0)[lb]{\smash{$c_2$}}}%
  \end{picture}%
\endgroup%

%% file: opp_orientation.pdf_tex
\begingroup%
  \makeatletter%
  \providecommand\color[2][]{%
    \errmessage{(Inkscape) Color is used for the text in Inkscape, but the package 'color.sty' is not loaded}%
    \renewcommand\color[2][]{}%
  }%
  \providecommand\transparent[1]{%
    \errmessage{(Inkscape) Transparency is used (non-zero) for the text in Inkscape, but the package 'transparent.sty' is not loaded}%
    \renewcommand\transparent[1]{}%
  }%
  \providecommand\rotatebox[2]{#2}%
  \ifx\svgwidth\undefined%
    \setlength{\unitlength}{874.16102132bp}%
    \ifx\svgscale\undefined%
      \relax%
    \else%
      \setlength{\unitlength}{\unitlength * \real{\svgscale}}%
    \fi%
  \else%
    \setlength{\unitlength}{\svgwidth}%
  \fi%
  \global\let\svgwidth\undefined%
  \global\let\svgscale\undefined%
  \makeatother%
  \begin{picture}(1,1.13654498)%
    \put(0,0){\includegraphics[width=\unitlength,page=1]{opp_orientation.pdf}}%
    \put(0.95068389,0.33120144){\color[rgb]{0,0,0}\makebox(0,0)[lb]{\smash{$a$}}}%
    \put(0.95068389,0.84369281){\color[rgb]{0,0,0}\makebox(0,0)[lb]{\smash{$a$}}}%
    \put(0.49310229,0.00174269){\color[rgb]{0,0,0}\makebox(0,0)[lb]{\smash{$b$}}}%
    \put(0,0){\includegraphics[width=\unitlength,page=2]{opp_orientation.pdf}}%
    \put(-0.00108584,0.40441455){\color[rgb]{0,0,0}\makebox(0,0)[lb]{\smash{$c_2$}}}%
    \put(0.95068389,0.77047981){\color[rgb]{0,0,0}\makebox(0,0)[lb]{\smash{$c_1$}}}%
  \end{picture}%
\endgroup%

%% file: bicorn1.pdf_tex
\begingroup%
  \makeatletter%
  \providecommand\color[2][]{%
    \errmessage{(Inkscape) Color is used for the text in Inkscape, but the package 'color.sty' is not loaded}%
    \renewcommand\color[2][]{}%
  }%
  \providecommand\transparent[1]{%
    \errmessage{(Inkscape) Transparency is used (non-zero) for the text in Inkscape, but the package 'transparent.sty' is not loaded}%
    \renewcommand\transparent[1]{}%
  }%
  \providecommand\rotatebox[2]{#2}%
  \ifx\svgwidth\undefined%
    \setlength{\unitlength}{952.53281481bp}%
    \ifx\svgscale\undefined%
      \relax%
    \else%
      \setlength{\unitlength}{\unitlength * \real{\svgscale}}%
    \fi%
  \else%
    \setlength{\unitlength}{\svgwidth}%
  \fi%
  \global\let\svgwidth\undefined%
  \global\let\svgscale\undefined%
  \makeatother%
  \begin{picture}(1,0.43482891)%
    \put(0,0){\includegraphics[width=\unitlength,page=1]{bicorn1.pdf}}%
    \put(0.0695118,0.24154895){\color[rgb]{0,0,0}\makebox(0,0)[lb]{\smash{$c$}}}%
    \put(0.96880824,0.38747253){\color[rgb]{0,0,0}\makebox(0,0)[lb]{\smash{$a$}}}%
    \put(0.96956101,0.00159935){\color[rgb]{0,0,0}\makebox(0,0)[lb]{\smash{$b$}}}%
  \end{picture}%
\endgroup%

%% file: bicorn2.pdf_tex
\begingroup%
  \makeatletter%
  \providecommand\color[2][]{%
    \errmessage{(Inkscape) Color is used for the text in Inkscape, but the package 'color.sty' is not loaded}%
    \renewcommand\color[2][]{}%
  }%
  \providecommand\transparent[1]{%
    \errmessage{(Inkscape) Transparency is used (non-zero) for the text in Inkscape, but the package 'transparent.sty' is not loaded}%
    \renewcommand\transparent[1]{}%
  }%
  \providecommand\rotatebox[2]{#2}%
  \ifx\svgwidth\undefined%
    \setlength{\unitlength}{1152.3608879bp}%
    \ifx\svgscale\undefined%
      \relax%
    \else%
      \setlength{\unitlength}{\unitlength * \real{\svgscale}}%
    \fi%
  \else%
    \setlength{\unitlength}{\svgwidth}%
  \fi%
  \global\let\svgwidth\undefined%
  \global\let\svgscale\undefined%
  \makeatother%
  \begin{picture}(1,0.7013749)%
    \put(0,0){\includegraphics[width=\unitlength,page=1]{bicorn2.pdf}}%
    \put(0.9523888,0.59439023){\color[rgb]{0,0,0}\makebox(0,0)[lb]{\smash{$c$}}}%
    \put(0.05584427,0.00132207){\color[rgb]{0,0,0}\makebox(0,0)[lb]{\smash{$a$}}}%
    \put(0.9186693,0.30876546){\color[rgb]{0,0,0}\makebox(0,0)[lb]{\smash{$b$}}}%
  \end{picture}%
\endgroup%

%% file: small_diam1.pdf_tex
\begingroup%
  \makeatletter%
  \providecommand\color[2][]{%
    \errmessage{(Inkscape) Color is used for the text in Inkscape, but the package 'color.sty' is not loaded}%
    \renewcommand\color[2][]{}%
  }%
  \providecommand\transparent[1]{%
    \errmessage{(Inkscape) Transparency is used (non-zero) for the text in Inkscape, but the package 'transparent.sty' is not loaded}%
    \renewcommand\transparent[1]{}%
  }%
  \providecommand\rotatebox[2]{#2}%
  \ifx\svgwidth\undefined%
    \setlength{\unitlength}{952.53281481bp}%
    \ifx\svgscale\undefined%
      \relax%
    \else%
      \setlength{\unitlength}{\unitlength * \real{\svgscale}}%
    \fi%
  \else%
    \setlength{\unitlength}{\svgwidth}%
  \fi%
  \global\let\svgwidth\undefined%
  \global\let\svgscale\undefined%
  \makeatother%
  \begin{picture}(1,1.09338722)%
    \put(0,0){\includegraphics[width=\unitlength,page=1]{small_diam1.pdf}}%
    \put(0.0695118,0.54973647){\color[rgb]{0,0,0}\makebox(0,0)[lb]{\smash{$c$}}}%
    \put(0.96880824,0.69566005){\color[rgb]{0,0,0}\makebox(0,0)[lb]{\smash{$b$}}}%
    \put(0.96956101,0.30978688){\color[rgb]{0,0,0}\makebox(0,0)[lb]{\smash{$a$}}}%
  \end{picture}%
\endgroup%

%% file: small_diam2.pdf_tex
\begingroup%
  \makeatletter%
  \providecommand\color[2][]{%
    \errmessage{(Inkscape) Color is used for the text in Inkscape, but the package 'color.sty' is not loaded}%
    \renewcommand\color[2][]{}%
  }%
  \providecommand\transparent[1]{%
    \errmessage{(Inkscape) Transparency is used (non-zero) for the text in Inkscape, but the package 'transparent.sty' is not loaded}%
    \renewcommand\transparent[1]{}%
  }%
  \providecommand\rotatebox[2]{#2}%
  \ifx\svgwidth\undefined%
    \setlength{\unitlength}{1675.71019931bp}%
    \ifx\svgscale\undefined%
      \relax%
    \else%
      \setlength{\unitlength}{\unitlength * \real{\svgscale}}%
    \fi%
  \else%
    \setlength{\unitlength}{\svgwidth}%
  \fi%
  \global\let\svgwidth\undefined%
  \global\let\svgscale\undefined%
  \makeatother%
  \begin{picture}(1,0.55605885)%
    \put(0,0){\includegraphics[width=\unitlength,page=1]{small_diam2.pdf}}%
    \put(0.7139061,0.40858882){\color[rgb]{0,0,0}\makebox(0,0)[lb]{\smash{$c$}}}%
    \put(0.85712897,0.50407072){\color[rgb]{0,0,0}\makebox(0,0)[lb]{\smash{$b$}}}%
    \put(0.98073555,0.01773236){\color[rgb]{0,0,0}\makebox(0,0)[lb]{\smash{$a$}}}%
  \end{picture}%
\endgroup%

%% file: same_sign1.pdf_tex
\begingroup%
  \makeatletter%
  \providecommand\color[2][]{%
    \errmessage{(Inkscape) Color is used for the text in Inkscape, but the package 'color.sty' is not loaded}%
    \renewcommand\color[2][]{}%
  }%
  \providecommand\transparent[1]{%
    \errmessage{(Inkscape) Transparency is used (non-zero) for the text in Inkscape, but the package 'transparent.sty' is not loaded}%
    \renewcommand\transparent[1]{}%
  }%
  \providecommand\rotatebox[2]{#2}%
  \ifx\svgwidth\undefined%
    \setlength{\unitlength}{951.68678103bp}%
    \ifx\svgscale\undefined%
      \relax%
    \else%
      \setlength{\unitlength}{\unitlength * \real{\svgscale}}%
    \fi%
  \else%
    \setlength{\unitlength}{\svgwidth}%
  \fi%
  \global\let\svgwidth\undefined%
  \global\let\svgscale\undefined%
  \makeatother%
  \begin{picture}(1,0.52486867)%
    \put(0,0){\includegraphics[width=\unitlength,page=1]{same_sign1.pdf}}%
    \put(0.91594287,0.33710477){\color[rgb]{0,0,0}\makebox(0,0)[lb]{\smash{$b$}}}%
    \put(0.92060495,0.0098168){\color[rgb]{0,0,0}\makebox(0,0)[lb]{\smash{$a$}}}%
    \put(0.06515392,0.20063532){\color[rgb]{0,0,0}\makebox(0,0)[lb]{\smash{$x$}}}%
    \put(0.86505053,0.19554044){\color[rgb]{0,0,0}\makebox(0,0)[lb]{\smash{$y$}}}%
    \put(0.36405372,0.38744767){\color[rgb]{0,0,0}\makebox(0,0)[lb]{\smash{$z$}}}%
    \put(0,0){\includegraphics[width=\unitlength,page=2]{same_sign1.pdf}}%
    \put(0.96694814,0.45028454){\color[rgb]{0,0,1}\makebox(0,0)[lb]{\smash{$c'$}}}%
    \put(0.18733655,0.29061184){\color[rgb]{1,0,0}\makebox(0,0)[lb]{\smash{$c$}}}%
  \end{picture}%
\endgroup%

%% file: same_sign2.pdf_tex
\begingroup%
  \makeatletter%
  \providecommand\color[2][]{%
    \errmessage{(Inkscape) Color is used for the text in Inkscape, but the package 'color.sty' is not loaded}%
    \renewcommand\color[2][]{}%
  }%
  \providecommand\transparent[1]{%
    \errmessage{(Inkscape) Transparency is used (non-zero) for the text in Inkscape, but the package 'transparent.sty' is not loaded}%
    \renewcommand\transparent[1]{}%
  }%
  \providecommand\rotatebox[2]{#2}%
  \ifx\svgwidth\undefined%
    \setlength{\unitlength}{1378.65770715bp}%
    \ifx\svgscale\undefined%
      \relax%
    \else%
      \setlength{\unitlength}{\unitlength * \real{\svgscale}}%
    \fi%
  \else%
    \setlength{\unitlength}{\svgwidth}%
  \fi%
  \global\let\svgwidth\undefined%
  \global\let\svgscale\undefined%
  \makeatother%
  \begin{picture}(1,0.62988637)%
    \put(0,0){\includegraphics[width=\unitlength,page=1]{same_sign2.pdf}}%
    \put(0.0930998,0.0696725){\color[rgb]{0,0,0}\makebox(0,0)[lb]{\smash{$x$}}}%
    \put(0.78942933,0.25536038){\color[rgb]{0,0,0}\makebox(0,0)[lb]{\smash{$y$}}}%
    \put(0.35505234,0.31836162){\color[rgb]{0,0,0}\makebox(0,0)[lb]{\smash{$z$}}}%
    \put(0,0){\includegraphics[width=\unitlength,page=2]{same_sign2.pdf}}%
    \put(0.96988751,0.59519591){\color[rgb]{0,0,1}\makebox(0,0)[lb]{\smash{$c'$}}}%
    \put(0.2619525,0.22299121){\color[rgb]{1,0,0}\makebox(0,0)[lb]{\smash{$c$}}}%
  \end{picture}%
\endgroup%

%% file: both_directions1a.pdf_tex
\begingroup%
  \makeatletter%
  \providecommand\color[2][]{%
    \errmessage{(Inkscape) Color is used for the text in Inkscape, but the package 'color.sty' is not loaded}%
    \renewcommand\color[2][]{}%
  }%
  \providecommand\transparent[1]{%
    \errmessage{(Inkscape) Transparency is used (non-zero) for the text in Inkscape, but the package 'transparent.sty' is not loaded}%
    \renewcommand\transparent[1]{}%
  }%
  \providecommand\rotatebox[2]{#2}%
  \ifx\svgwidth\undefined%
    \setlength{\unitlength}{981.07045091bp}%
    \ifx\svgscale\undefined%
      \relax%
    \else%
      \setlength{\unitlength}{\unitlength * \real{\svgscale}}%
    \fi%
  \else%
    \setlength{\unitlength}{\svgwidth}%
  \fi%
  \global\let\svgwidth\undefined%
  \global\let\svgscale\undefined%
  \makeatother%
  \begin{picture}(1,0.61100181)%
    \put(0,0){\includegraphics[width=\unitlength,page=1]{both_directions1a.pdf}}%
    \put(0.09730891,0.23452274){\color[rgb]{1,0,0}\makebox(0,0)[lb]{\smash{$c$}}}%
    \put(0.90949141,0.39102849){\color[rgb]{0,0,0}\makebox(0,0)[lb]{\smash{$b$}}}%
    \put(0.96952971,0.00649519){\color[rgb]{0,0,0}\makebox(0,0)[lb]{\smash{$a$}}}%
    \put(0,0){\includegraphics[width=\unitlength,page=2]{both_directions1a.pdf}}%
    \put(0.32253461,0.52388244){\color[rgb]{0.53333333,0.53333333,0}\makebox(0,0)[lb]{\smash{$c_2$}}}%
    \put(0.7904044,0.59636931){\color[rgb]{0,0,1}\makebox(0,0)[lb]{\smash{$c_1$}}}%
    \put(0.7393341,-0.06919191){\color[rgb]{0,0,0}\makebox(0,0)[lb]{\smash{}}}%
  \end{picture}%
\endgroup%

%% file: both_directions1b.pdf_tex
\begingroup%
  \makeatletter%
  \providecommand\color[2][]{%
    \errmessage{(Inkscape) Color is used for the text in Inkscape, but the package 'color.sty' is not loaded}%
    \renewcommand\color[2][]{}%
  }%
  \providecommand\transparent[1]{%
    \errmessage{(Inkscape) Transparency is used (non-zero) for the text in Inkscape, but the package 'transparent.sty' is not loaded}%
    \renewcommand\transparent[1]{}%
  }%
  \providecommand\rotatebox[2]{#2}%
  \ifx\svgwidth\undefined%
    \setlength{\unitlength}{974.47355015bp}%
    \ifx\svgscale\undefined%
      \relax%
    \else%
      \setlength{\unitlength}{\unitlength * \real{\svgscale}}%
    \fi%
  \else%
    \setlength{\unitlength}{\svgwidth}%
  \fi%
  \global\let\svgwidth\undefined%
  \global\let\svgscale\undefined%
  \makeatother%
  \begin{picture}(1,0.61513798)%
    \put(0,0){\includegraphics[width=\unitlength,page=1]{both_directions1b.pdf}}%
    \put(0.09162479,0.23611029){\color[rgb]{1,0,0}\makebox(0,0)[lb]{\smash{$c$}}}%
    \put(0.91916287,0.39261646){\color[rgb]{0,0,0}\makebox(0,0)[lb]{\smash{$b$}}}%
    \put(0.97140887,0.00156329){\color[rgb]{0,0,0}\makebox(0,0)[lb]{\smash{$a$}}}%
    \put(0.78941233,0.60040649){\color[rgb]{0,0,1}\makebox(0,0)[lb]{\smash{$c_1$}}}%
    \put(0.73799631,-0.06966046){\color[rgb]{0,0,0}\makebox(0,0)[lb]{\smash{}}}%
    \put(0,0){\includegraphics[width=\unitlength,page=2]{both_directions1b.pdf}}%
    \put(0.06370176,0.56736128){\color[rgb]{0.53333333,0.53333333,0}\makebox(0,0)[lb]{\smash{$c_2$}}}%
  \end{picture}%
\endgroup%

%% file: both_directions2a.pdf_tex
\begingroup%
  \makeatletter%
  \providecommand\color[2][]{%
    \errmessage{(Inkscape) Color is used for the text in Inkscape, but the package 'color.sty' is not loaded}%
    \renewcommand\color[2][]{}%
  }%
  \providecommand\transparent[1]{%
    \errmessage{(Inkscape) Transparency is used (non-zero) for the text in Inkscape, but the package 'transparent.sty' is not loaded}%
    \renewcommand\transparent[1]{}%
  }%
  \providecommand\rotatebox[2]{#2}%
  \ifx\svgwidth\undefined%
    \setlength{\unitlength}{1372.7786375bp}%
    \ifx\svgscale\undefined%
      \relax%
    \else%
      \setlength{\unitlength}{\unitlength * \real{\svgscale}}%
    \fi%
  \else%
    \setlength{\unitlength}{\svgwidth}%
  \fi%
  \global\let\svgwidth\undefined%
  \global\let\svgscale\undefined%
  \makeatother%
  \begin{picture}(1,0.57274623)%
    \put(0,0){\includegraphics[width=\unitlength,page=1]{both_directions2a.pdf}}%
    \put(0.16026721,0.08682159){\color[rgb]{0,0,1}\makebox(0,0)[lb]{\smash{$c_1$}}}%
    \put(0.79486008,0.34701646){\color[rgb]{0.53333333,0.53333333,0}\makebox(0,0)[lb]{\smash{$c_2$}}}%
    \put(0.33157198,0.21574167){\color[rgb]{1,0,0}\makebox(0,0)[lb]{\smash{$c$}}}%
    \put(0.72892834,0.35643527){\color[rgb]{0,0,0}\makebox(0,0)[lb]{\smash{$b$}}}%
    \put(0.97970443,0.31169587){\color[rgb]{0,0,0}\makebox(0,0)[lb]{\smash{$a$}}}%
  \end{picture}%
\endgroup%

%% file: both_directions2b.pdf_tex
\begingroup%
  \makeatletter%
  \providecommand\color[2][]{%
    \errmessage{(Inkscape) Color is used for the text in Inkscape, but the package 'color.sty' is not loaded}%
    \renewcommand\color[2][]{}%
  }%
  \providecommand\transparent[1]{%
    \errmessage{(Inkscape) Transparency is used (non-zero) for the text in Inkscape, but the package 'transparent.sty' is not loaded}%
    \renewcommand\transparent[1]{}%
  }%
  \providecommand\rotatebox[2]{#2}%
  \ifx\svgwidth\undefined%
    \setlength{\unitlength}{1407.78177096bp}%
    \ifx\svgscale\undefined%
      \relax%
    \else%
      \setlength{\unitlength}{\unitlength * \real{\svgscale}}%
    \fi%
  \else%
    \setlength{\unitlength}{\svgwidth}%
  \fi%
  \global\let\svgwidth\undefined%
  \global\let\svgscale\undefined%
  \makeatother%
  \begin{picture}(1,0.5503467)%
    \put(0,0){\includegraphics[width=\unitlength,page=1]{both_directions2b.pdf}}%
    \put(0.9185112,0.46277278){\color[rgb]{1,0,0}\makebox(0,0)[lb]{\smash{$c$}}}%
    \put(0.98020908,0.29553907){\color[rgb]{0,0,0}\makebox(0,0)[lb]{\smash{$a$}}}%
    \put(0.7301704,0.34100059){\color[rgb]{0,0,0}\makebox(0,0)[lb]{\smash{$b$}}}%
    \put(0,0){\includegraphics[width=\unitlength,page=2]{both_directions2b.pdf}}%
    \put(0.22764774,0.32546713){\color[rgb]{0.53333333,0.53333333,0}\makebox(0,0)[lb]{\smash{$c_2$}}}%
    \put(0.1415645,0.15586834){\color[rgb]{0,0,1}\makebox(0,0)[lb]{\smash{$c_1$}}}%
    \put(0,0){\includegraphics[width=\unitlength,page=3]{both_directions2b.pdf}}%
  \end{picture}%
\endgroup%

%% file: nonsep2a.pdf_tex
\begingroup%
  \makeatletter%
  \providecommand\color[2][]{%
    \errmessage{(Inkscape) Color is used for the text in Inkscape, but the package 'color.sty' is not loaded}%
    \renewcommand\color[2][]{}%
  }%
  \providecommand\transparent[1]{%
    \errmessage{(Inkscape) Transparency is used (non-zero) for the text in Inkscape, but the package 'transparent.sty' is not loaded}%
    \renewcommand\transparent[1]{}%
  }%
  \providecommand\rotatebox[2]{#2}%
  \ifx\svgwidth\undefined%
    \setlength{\unitlength}{1350.94776544bp}%
    \ifx\svgscale\undefined%
      \relax%
    \else%
      \setlength{\unitlength}{\unitlength * \real{\svgscale}}%
    \fi%
  \else%
    \setlength{\unitlength}{\svgwidth}%
  \fi%
  \global\let\svgwidth\undefined%
  \global\let\svgscale\undefined%
  \makeatother%
  \begin{picture}(1,0.5906836)%
    \put(0,0){\includegraphics[width=\unitlength,page=1]{nonsep2a.pdf}}%
    \put(0.06409342,0.1779177){\color[rgb]{0,0.70588235,0.56470588}\makebox(0,0)[lb]{\smash{$e_2$}}}%
    \put(0.17826031,0.09690661){\color[rgb]{0,0,1}\makebox(0,0)[lb]{\smash{$c_1$}}}%
    \put(0,0){\includegraphics[width=\unitlength,page=2]{nonsep2a.pdf}}%
    \put(0.39480471,0.35293148){\color[rgb]{0,1,0}\makebox(0,0)[lb]{\smash{$c'$}}}%
  \end{picture}%
\endgroup%

%% file: nonsep2b.pdf_tex
\begingroup%
  \makeatletter%
  \providecommand\color[2][]{%
    \errmessage{(Inkscape) Color is used for the text in Inkscape, but the package 'color.sty' is not loaded}%
    \renewcommand\color[2][]{}%
  }%
  \providecommand\transparent[1]{%
    \errmessage{(Inkscape) Transparency is used (non-zero) for the text in Inkscape, but the package 'transparent.sty' is not loaded}%
    \renewcommand\transparent[1]{}%
  }%
  \providecommand\rotatebox[2]{#2}%
  \ifx\svgwidth\undefined%
    \setlength{\unitlength}{1370.9265329bp}%
    \ifx\svgscale\undefined%
      \relax%
    \else%
      \setlength{\unitlength}{\unitlength * \real{\svgscale}}%
    \fi%
  \else%
    \setlength{\unitlength}{\svgwidth}%
  \fi%
  \global\let\svgwidth\undefined%
  \global\let\svgscale\undefined%
  \makeatother%
  \begin{picture}(1,0.58200794)%
    \put(0,0){\includegraphics[width=\unitlength,page=1]{nonsep2b.pdf}}%
    \put(0.16214571,0.1769246){\color[rgb]{0,0,1}\makebox(0,0)[lb]{\smash{$c_1$}}}%
    \put(0,0){\includegraphics[width=\unitlength,page=2]{nonsep2b.pdf}}%
    \put(0.07925946,0.16187597){\color[rgb]{0,0.70588235,0.56470588}\makebox(0,0)[lb]{\smash{$e_2$}}}%
    \put(0,0){\includegraphics[width=\unitlength,page=3]{nonsep2b.pdf}}%
    \put(0.4025563,0.377667){\color[rgb]{0,1,0}\makebox(0,0)[lb]{\smash{$c'$}}}%
  \end{picture}%
\endgroup%

%% file: b_d_bicorn.pdf_tex
\begingroup%
  \makeatletter%
  \providecommand\color[2][]{%
    \errmessage{(Inkscape) Color is used for the text in Inkscape, but the package 'color.sty' is not loaded}%
    \renewcommand\color[2][]{}%
  }%
  \providecommand\transparent[1]{%
    \errmessage{(Inkscape) Transparency is used (non-zero) for the text in Inkscape, but the package 'transparent.sty' is not loaded}%
    \renewcommand\transparent[1]{}%
  }%
  \providecommand\rotatebox[2]{#2}%
  \ifx\svgwidth\undefined%
    \setlength{\unitlength}{1016.35040523bp}%
    \ifx\svgscale\undefined%
      \relax%
    \else%
      \setlength{\unitlength}{\unitlength * \real{\svgscale}}%
    \fi%
  \else%
    \setlength{\unitlength}{\svgwidth}%
  \fi%
  \global\let\svgwidth\undefined%
  \global\let\svgscale\undefined%
  \makeatother%
  \begin{picture}(1,0.59936819)%
    \put(0,0){\includegraphics[width=\unitlength,page=1]{b_d_bicorn.pdf}}%
    \put(0.90797587,0.46293454){\color[rgb]{0,0,0}\makebox(0,0)[lb]{\smash{$b$}}}%
    \put(0.90868137,0.1012907){\color[rgb]{0,0,0}\makebox(0,0)[lb]{\smash{$a$}}}%
    \put(0,0){\includegraphics[width=\unitlength,page=2]{b_d_bicorn.pdf}}%
    \put(0.26087738,0.46187162){\color[rgb]{0,0,0}\makebox(0,0)[lb]{\smash{$y_1$}}}%
    \put(0.45765996,0.47424081){\color[rgb]{0,0,0}\makebox(0,0)[lb]{\smash{$y_2$}}}%
    \put(0.63280467,0.45869113){\color[rgb]{0,0,0}\makebox(0,0)[lb]{\smash{$y_3$}}}%
    \put(0.4209382,0.49364136){\color[rgb]{0,0,0}\makebox(0,0)[lb]{\smash{$y_4$}}}%
    \put(0.35767156,0.49441309){\color[rgb]{0,0,0}\makebox(0,0)[lb]{\smash{$y_5$}}}%
    \put(0.66631179,0.48145009){\color[rgb]{0,0,0}\makebox(0,0)[lb]{\smash{$y_6$}}}%
    \put(0.76172636,0.45123547){\color[rgb]{0,0,0}\makebox(0,0)[lb]{\smash{$y_7$}}}%
    \put(0.68873885,0.40930765){\color[rgb]{0,0,1}\makebox(0,0)[lb]{\smash{$c_6''$}}}%
    \put(0.25125839,0.08905358){\color[rgb]{0,0,0}\makebox(0,0)[lb]{\smash{$d$}}}%
    \put(0.31963883,0.1812877){\color[rgb]{1,0,0}\makebox(0,0)[lb]{\smash{$c'$}}}%
    \put(0,0){\includegraphics[width=\unitlength,page=3]{b_d_bicorn.pdf}}%
    \put(0.56448476,0.49701645){\color[rgb]{0,0,1}\makebox(0,0)[lb]{\smash{$c_2''$}}}%
  \end{picture}%
\endgroup%

%% file: modified.pdf_tex
\begingroup%
  \makeatletter%
  \providecommand\color[2][]{%
    \errmessage{(Inkscape) Color is used for the text in Inkscape, but the package 'color.sty' is not loaded}%
    \renewcommand\color[2][]{}%
  }%
  \providecommand\transparent[1]{%
    \errmessage{(Inkscape) Transparency is used (non-zero) for the text in Inkscape, but the package 'transparent.sty' is not loaded}%
    \renewcommand\transparent[1]{}%
  }%
  \providecommand\rotatebox[2]{#2}%
  \ifx\svgwidth\undefined%
    \setlength{\unitlength}{691.12788643bp}%
    \ifx\svgscale\undefined%
      \relax%
    \else%
      \setlength{\unitlength}{\unitlength * \real{\svgscale}}%
    \fi%
  \else%
    \setlength{\unitlength}{\svgwidth}%
  \fi%
  \global\let\svgwidth\undefined%
  \global\let\svgscale\undefined%
  \makeatother%
  \begin{picture}(1,0.22736518)%
    \put(0,0){\includegraphics[width=\unitlength,page=1]{modified.pdf}}%
    \put(-0.00152604,0.12639391){\color[rgb]{0,0,0}\makebox(0,0)[lb]{\smash{$\alpha$}}}%
    \put(0.37462196,0.2065941){\color[rgb]{0,0,0}\makebox(0,0)[lb]{\smash{$d$}}}%
    \put(0,0){\includegraphics[width=\unitlength,page=2]{modified.pdf}}%
    \put(0.95968717,0.11772361){\color[rgb]{1,0,0}\makebox(0,0)[lb]{\smash{$c$}}}%
    \put(-0.0014331,0.09959474){\color[rgb]{0.53333333,0.53333333,0}\makebox(0,0)[lb]{\smash{$c_0$}}}%
    \put(0,0){\includegraphics[width=\unitlength,page=3]{modified.pdf}}%
    \put(0.48698863,0.03321699){\color[rgb]{0,0,1}\makebox(0,0)[lb]{\smash{$c_i'''$}}}%
  \end{picture}%
\endgroup%

%% file: c_cprime_intersection.pdf_tex
\begingroup%
  \makeatletter%
  \providecommand\color[2][]{%
    \errmessage{(Inkscape) Color is used for the text in Inkscape, but the package 'color.sty' is not loaded}%
    \renewcommand\color[2][]{}%
  }%
  \providecommand\transparent[1]{%
    \errmessage{(Inkscape) Transparency is used (non-zero) for the text in Inkscape, but the package 'transparent.sty' is not loaded}%
    \renewcommand\transparent[1]{}%
  }%
  \providecommand\rotatebox[2]{#2}%
  \newcommand*\fsize{\dimexpr\f@size pt\relax}%
  \newcommand*\lineheight[1]{\fontsize{\fsize}{#1\fsize}\selectfont}%
  \ifx\svgwidth\undefined%
    \setlength{\unitlength}{951.78141917bp}%
    \ifx\svgscale\undefined%
      \relax%
    \else%
      \setlength{\unitlength}{\unitlength * \real{\svgscale}}%
    \fi%
  \else%
    \setlength{\unitlength}{\svgwidth}%
  \fi%
  \global\let\svgwidth\undefined%
  \global\let\svgscale\undefined%
  \makeatother%
  \begin{picture}(1,0.55049883)%
    \lineheight{1}%
    \setlength\tabcolsep{0pt}%
    \put(0,0){\includegraphics[width=\unitlength,page=1]{c_cprime_intersection.pdf}}%
    \put(0.96953692,0.41731741){\color[rgb]{0,0,0}\makebox(0,0)[lt]{\lineheight{0}\smash{\begin{tabular}[t]{l}$b$\end{tabular}}}}%
    \put(0.97029029,0.03113961){\color[rgb]{0,0,0}\makebox(0,0)[lt]{\lineheight{0}\smash{\begin{tabular}[t]{l}$a$\end{tabular}}}}%
    \put(0,0){\includegraphics[width=\unitlength,page=2]{c_cprime_intersection.pdf}}%
    \put(0.07738238,0.2141285){\color[rgb]{1,0,0}\makebox(0,0)[lt]{\lineheight{1.25}\smash{\begin{tabular}[t]{l}$c_0$\end{tabular}}}}%
    \put(0,0){\includegraphics[width=\unitlength,page=3]{c_cprime_intersection.pdf}}%
    \put(0.40184185,0.05581096){\color[rgb]{0.53333333,0.53333333,0}\makebox(0,0)[lt]{\lineheight{1.25}\smash{\begin{tabular}[t]{l}$c'$\end{tabular}}}}%
    \put(0.22060281,0.44530612){\color[rgb]{0,0,0}\makebox(0,0)[lt]{\lineheight{1.25}\smash{\begin{tabular}[t]{l}$y_1$\end{tabular}}}}%
    \put(0.80822269,0.414912){\color[rgb]{0,0,0}\makebox(0,0)[lt]{\lineheight{1.25}\smash{\begin{tabular}[t]{l}$y_n$\end{tabular}}}}%
  \end{picture}%
\endgroup%